\documentclass[11pt,a4paper,dvipsnames]{article}
\usepackage[utf8]{inputenc}

\title{On the Lasso for Graphical Continuous Lyapunov Models}
\author{Philipp Dettling, Mathias Drton, Mladen Kolar}
\date{}
\usepackage{a4wide}
\usepackage[margin=1cm]{caption}
\usepackage{subcaption}
\usepackage[top=3cm,bottom=2cm,left=3cm,right=3cm,marginparwidth=1.75cm]{geometry}
\usepackage{graphicx}
\graphicspath{{Figures/}}
\usepackage{xcolor}
\usepackage{tikz,tikz-3dplot,tikz-cd}

\usepackage{times}

\usepackage{amsmath, amssymb} %, mathbbol}
\usepackage[shortlabels]{enumitem}
\usepackage{graphicx}
\usepackage{makecell}
\usepackage[ruled]{algorithm2e}
\usepackage{multicol}
\usepackage[section]{placeins}
\usepackage{mwe}
\usepackage{tikz}
\usepackage{array,amscd}
\usepackage{amsfonts}
\usepackage[T1]{fontenc} 
\usepackage{enumerate}
\usepackage{appendix}
\usepackage[arrow, matrix, curve]{xy}
\usepackage[round,sort,authoryear]{natbib}
\usepackage{booktabs}
\usepackage{siunitx}
\usepackage{multirow}
\usepackage[font=small,labelfont=bf]{caption}
\usepackage{blkarray}
\usepackage{dsfont}
\usepackage{longtable}

\usepackage{tikz}
\usetikzlibrary {shapes}
\usetikzlibrary {arrows}
\usetikzlibrary {positioning}
\usetikzlibrary {calc}

\usepackage{enumitem}
\setenumerate{leftmargin=*}
\allowdisplaybreaks[3]

\newcommand{\RR}{\mathbb{R}}

\newcommand{\PD}{\mathrm{PD}}

\definecolor{darkgreen}{rgb}{0,0.4,0}
\definecolor{MyBlue}{rgb}{0,0.08,0.7} 
\definecolor{MyRed}{rgb}{0.85,0.08,0}

\makeatletter
\renewcommand*\env@matrix[1][\arraystretch]{%
  \edef\arraystretch{#1}%
  \hskip -\arraycolsep
  \let\@ifnextchar\new@ifnextchar
  \array{*\c@MaxMatrixCols c}}
\makeatother

\newcommand{\norm}[1]{\|#1\|}

\newcommand{\opnorm}[2]{| \! | \! | #1 | \! | \! |_{{#2}}}

\newcommand{\DM}{\Delta_{M}}
\newcommand{\tDM}{\tilde\Delta_{M}}
\newcommand{\DGamma}{\Delta_{\Gamma}}
\newcommand{\Dg}{\Delta_{g}}
\newcommand{\hGamma}{\hat\Gamma}
\newcommand{\hg}{\hat g}
\newcommand{\sGamma}{\Gamma^*}
\newcommand{\sg}{g^*}

\newcommand{\DSigma}{\Delta_\Sigma}
\newcommand{\hSigma}{\hat \Sigma}
\newcommand{\sSigma}{\Sigma^*}

\newcommand{\vtM}{{\rm vec}(\tilde M)}
\newcommand{\vsM}{{\rm vec}(M^*)}
\newcommand{\vhM}{{\rm vec}(\hat M)}
\newcommand{\vM}{{\rm vec}(M)}

\newcommand{\sign}{{\rm sign}}
\newcommand{\diag}{{\rm diag}}

\newcommand{\rbr}[1]{\left(#1\right)}

\numberwithin{equation}{section}

\RequirePackage[OT1]{fontenc}
\RequirePackage{amsmath}

\definecolor{mygreen}{RGB}{37,170,37}
\definecolor{myteal}{RGB}{6,144,179}
\definecolor{myblue}{RGB}{59,100,189}
\definecolor{myorange}{RGB}{240,153,85}
\definecolor{myred}{RGB}{190,15,15}
\definecolor{myyellow}{RGB}{255,220,130}
\definecolor{mypurple}{RGB}{139,0,139}

\usetikzlibrary{decorations.pathreplacing,decorations.markings}
\tikzset{%
  symbol/.style={
    draw=none,
    every to/.append style={
      edge node={node [sloped, allow upside down, auto=false]{$#1$}}
    },
  },
  on each segment/.style={
    decorate,
    decoration={
      show path construction,
      moveto code={},
      lineto code={
        \path [#1]
        (\tikzinputsegmentfirst) -- (\tikzinputsegmentlast);
      },
      curveto code={
        \path [#1] (\tikzinputsegmentfirst)
        .. controls
        (\tikzinputsegmentsupporta) and (\tikzinputsegmentsupportb)
        ..
        (\tikzinputsegmentlast);
      },
      closepath code={
        \path [#1]
        (\tikzinputsegmentfirst) -- (\tikzinputsegmentlast);
      },
    },
  },
  mid arrow/.style={postaction={decorate,decoration={
        markings,
        mark=at position .5 with {\arrow[#1]{stealth}}
      }}},
}

\usepackage{stmaryrd}
\usepackage[hypertexnames=false,pdftex,backref=page,
pdfpagemode=UseNone,
breaklinks=true,
extension=pdf,
colorlinks=true,
linkcolor=myblue,
citecolor=myblue,
urlcolor=myblue]{hyperref}

\usepackage{etoolbox}
\patchcmd{\thebibliography}
  {\settowidth}
  {\setlength{\itemsep}{0pt plus 0.1pt}\settowidth}
  {}{}
\apptocmd{\thebibliography}
  {\small}
  {}{}

\newtheorem{theorem}{Theorem}[section]
\newtheorem{corollary}[theorem]{Corollary}
\newtheorem{lemma}[theorem]{Lemma}
\newtheorem{proposition}[theorem]{Proposition}

\newtheorem{example}[theorem]{Example}
\newtheorem{definition}[theorem]{Definition}
\newtheorem{remark}[theorem]{Remark}

\newenvironment{proof}{\paragraph{Proof:}}{\hfill$\square$}

\begin{document}

\maketitle

\begin{abstract}
Graphical continuous Lyapunov models offer a new perspective on modeling causally interpretable dependence structure in multivariate data by treating each independent observation as a one-time cross-sectional snapshot of a temporal process. Specifically, the models assume that the observations are cross-sections of independent multivariate Ornstein-Uhlenbeck processes in equilibrium. The Gaussian equilibrium exists under a stability assumption on the drift matrix, and the equilibrium covariance matrix is determined by the continuous Lyapunov equation. Each graphical continuous Lyapunov model assumes the drift matrix to be sparse, with a support determined by a directed graph. A natural approach to model selection in this setting is to use an $\ell_1$-regularization technique that, based on a given sample covariance matrix, seeks to find a sparse approximate solution to the Lyapunov equation. We study the model selection properties of the resulting lasso technique to arrive at a consistency result.  Our detailed analysis reveals that the involved irrepresentability condition is surprisingly difficult to satisfy. While this may prevent asymptotic consistency in model selection, our numerical experiments indicate that even if the theoretical requirements for consistency are not met, the lasso approach is able to recover relevant structure of the drift matrix and is robust to aspects of model misspecification. 
\end{abstract}


\section{Introduction}

Directed graphical models are powerful tools for exploring cause-effect relationships in multivariate data \citep{pearl:2009,spirtes:2000,peters2017}. The causal aspect of the models is built on the assumption that each variable is a function of parent variables and independent noise. This approach is also known as structural causal modeling. For directed acyclic graphs (DAGs), the resulting models have simple interpretations and statistically favorable density factorization properties that facilitate large-scale analyses \citep{maathuis2019}. The situation is more complicated when the graph is allowed to contain directed cycles, which represent feedback loops \citep{bongers2021}. Although a model can still be defined by solving structural equations, directed cycles prevent density factorization, making it more difficult to perform tasks such as computation of maximum likelihood estimates \citep{drton:fox:wang2019} or model selection \cite[e.g.,][]{richardson1996,amendola2020}, even in the case of linear models. Importantly, the interpretation of the models also becomes more involved and typically appeals to dynamic processes in a post-hoc way. For example, \cite{fischer1970} provided an interpretation based on data averaged over time, while alternative interpretations in terms of differential equations were suggested by \cite{mooij2013} and \cite{bongers2018}.

\cite{katie2019} and \cite{hansen2020} proposed a different perspective, in which the distribution of an observed sample $X_1,\dots,X_n \in \mathbb{R}^p$, independent and identically distributed, is modeled through a temporal process in equilibrium.  Specifically, each random vector $X_i$ is assumed to be a single cross-sectional observation of a multivariate Ornstein-Uhlenbeck process (a multivariate continuous-time autoregressive process), which leads to $X_i$ being multivariate normal. We emphasize that the $n$ observations are obtained from $n$ independent processes. This setup is suitable, in particular, for applications in biology where each of the $n$ independent organisms may live up to the time of measurement (e.g., for gene expression) but is sacrificed for the measurement.

The $p$-dimensional Ornstein-Uhlenbeck process is the solution to the stochastic differential equation
\begin{equation}
\label{eq:pdimOU}
    \mathrm{d}X(t)=M(X(t)-a)\,\mathrm{d}t+ D\, \mathrm{d}W(t),
\end{equation}
where $W(t)$ is a Wiener process, and $a \in \mathbb{R}^{p}$ and $M,D \in \mathbb{R}^{p \times p}$ are non-singular parameter matrices. The \emph{drift matrix} $M$ is the key object of interest in the work of \cite{katie2019} and \cite{hansen2020} as it determines the causal relations between the coordinates of the Ornstein-Uhlenbeck process $X(t)$; see also \cite{Mogensen2018CausalLF}.  Provided $M$ is stable (i.e., all eigenvalues have a strictly negative real part), $X(t)$ admits an equilibrium distribution that is multivariate normal with a positive definite covariance matrix. This covariance matrix, denoted by $\Sigma$, is determined as the unique matrix that solves the Lyapunov equation
\begin{equation}
\label{eq:lyapunoveq}
    M\Sigma+\Sigma M^{\top}+C=0,
\end{equation}
where $C=DD^{\top}$. The vector $a$ is the mean vector of the equilibrium distribution.  Subsequently, we will assume without loss of generality that $a=0$, i.e., our observations are centered. Moreover, we will focus on the case where the positive definite volatility matrix $C$ is known up to a scalar multiple.  Since scaling $(M,C)$ does not change the solution $\Sigma$ in~\eqref{eq:lyapunoveq}, this case can be studied by reducing to the setting where $C$ is fully known; see Remark~\ref{rem:CID} for a further detailed discussion.  

Our interest is now in the selection of models that postulate that the drift matrix $M=(M_{ij})$ is sparse. In other words, we consider the estimation of the sparsity pattern (or support) of the drift matrix $M$.  This support is naturally represented by a directed graph $G=(V,E)$ with a vertex set $V=\{ 1, \dots , p \}$ and an edge set $E$ that includes the edge $i \to j$ precisely when $M_{ji}\not=0$. A stable matrix $M$ will have negative diagonal entries, and therefore the edge set $E$ will always contain all self-loops $i\to i$. However, we will not draw the self-loops in figures showing graphs.
\begin{example}
\label{exa:graphM}
The graph $G=(\lbrace 1,2,3 \rbrace, \lbrace 1 \to 2,2 \to 3, 1\to 1, 2 \to 2, 3\to3 \rbrace)$, shown in Figure~\ref{fig:examplegraph}, corresponds to the support of the matrix
$$M = \begin{pmatrix}
m_{11} & 0 & 0 \\
m_{21} & m_{22} & 0 \\
0 & m_{32} & m_{33}
\end{pmatrix}.$$
\end{example}

\begin{figure}[t]
\begin{center}
\medskip
\begin{tikzpicture}[->,every node/.style={circle,draw},line width=1pt, node distance=1.5cm]
%\hspace*{-1cm}
  \node (1)  {$1$};
  \node (2) [right of=1]{$2$};
  \node (3) [right of=2] {$3$};
\foreach \from/\to in {1/2,2/3}
\draw (\from) -- (\to);   
 \end{tikzpicture}  
\end{center}
\caption{Directed graph on 3 nodes.}
\label{fig:examplegraph}
\end{figure}

We remark that \cite{young2019} considered a related setup with discrete-time vector autoregressive (VAR) processes, which leads to the discrete Lyapunov equation.

\subsection{Support Recovery with the Direct Lyapunov Lasso}

In this paper, we will study an $\ell_1$-regularization method to estimate the support of the drift matrix $M$ from an i.i.d.~sample consisting of centered observations $X_1,\dots,X_n\in\mathbb{R}^p$. Let
\begin{equation}
\label{eq:samplecov}
    \hat{\Sigma} = \hat{\Sigma}^{(n)} = \frac{1}{n} \sum_{i = 1}^{n} X_i X_i^{\top}
\end{equation}
be the sample covariance matrix. The \emph{Direct Lyapunov Lasso} finds a sparse estimate of $M$ as a solution of the convex optimization problem
\begin{equation}
\label{eq:Frobeniuseq}
\min_{M \in \mathbb{R}^{p \times p}}
\frac{1}{2}\norm{M \hat{\Sigma} + \hat{\Sigma} M^{\top} + C}_{F}^{2} + \lambda \norm{M}_{1}
\end{equation}
with tuning parameter $\lambda>0$. This method is considered in numerical experiments by \cite{katie2019} as well as by \cite{hansen2020} who additionally explore non-convex methods based on regularizing Gaussian likelihood or a Frobenius loss. The Direct Lyapunov Lasso yields matrices in $\mathbb{R}^{p \times p}$ that can be non-stable.  If a stable estimate is required in such a case, one can appeal to projection onto the set of stable matrices; e.g., using techniques by \cite{Noferini2021}.

\subsection{Organization of the Paper}

We connect the Direct Lyapunov Lasso to more standard lasso problems by vectorizing the Lyapunov equation and describing the structure of the Hessian matrix for the smooth part of the Direct Lyapunov Lasso objective (Section~\ref{sec:hessian}).  In Section~\ref{sec:bounderror}, we present a result of statistical consistency (Theorem~\ref{thm:probsupport}), where the solution $\hat{M}$ of \eqref{eq:Frobeniuseq} is shown to converge in the max norm at a rate $\|\hat{M}-M^{*}\|_{\infty}=O(\sqrt{(dp)/n})$ with $d$ being the number of nonzero entries in the true drift matrix $M^{*}$.
Theorem \ref{thm:probsupport} only holds under an irrepresentability condition, which turns out to be more subtle than in the classical lasso regression. As we explore in Section~\ref{sec:irr}, the condition is highly dependent on the structure of the graph associated with the true signal.  In Section~\ref{sec:robust}, we present large-scale simulations (up to $p=50$) where the performance of the Direct Lyapunov Lasso on synthetic data is measured. Despite the theoretical requirements for consistency not being met, we observe performance at useful levels across various metrics. Finally, in Section~\ref{sec:realworlddata}, we apply the Direct Lyapunov Lasso to obtain an estimate of a protein signaling network that recovers important connections in a network that was previously reported as a gold standard.

\subsection{Motivating Example}

Before developing a detailed analysis of the Direct Lyapunov Lasso, we present an example that illustrates the behavior of estimates for growing sample size and highlights the impact of the irrepresentability condition.  We defer some of the details of how the example is designed to Appendix~\ref{sec:appendixsupplementexample}. 
\begin{example}
\label{exa:pathcycle}
Let $G_1$ be the directed path from 1 to 5, and let $G_2$ be the 5-cycle obtained by adding the edge $5\to 1$; see Figure~\ref{tab:intro}. For $G_1$ we define a (well-conditioned) stable matrix $M_1^{*}$ by setting the diagonal to $(-2,-3,-4,-5,-6)$ and the four nonzero subdiagonal entries to $0.65$. For $G_2$, we consider two cases.  First, we add the fixed entry $m_{15}=0.65$ to $M_1^{*}$ to obtain the matrix $M_2^{*}$. Second, we similarly include $m_{15}$ but select it randomly (uniform on $[0.5,1]$) in 100 instances. Using the Lyapunov equation \eqref{eq:lyapunoveq} with $C=2I_{p}$, for each setting and each one of 6 different sample sizes $n$ we simulate 100 Gaussian datasets. We also consider $n=\infty$, i.e., taking the population covariance matrices as input to the method. We calculate solutions to the Direct Lyapunov Lasso \eqref{eq:Frobeniuseq} for 100 choices of the regularization parameter $\lambda$.  From these solutions we compute the maximum accuracy, the maximum $\text{F}_1$-score, and the area under the ROC curve. Figure~\ref{fig:SR1} plots the performance measures, averaged over the 100 datasets in each pairing of setup and sample size. There the blue curves refer to $G_1$, the red curves to $G_2$ with $m_{15}=0.65$ fixed, and the green curves to $G_2$ with $m_{15}$ chosen randomly. We observe that for every sample size and performance measure, the Direct Lyapunov Lasso performs better for the path $G_1$ than for the cycle $G_2$. When the sample size is $n=10^{4}$, we observe an almost perfect recovery of $G_1$. However, increasing the sample size when recovering $G_2$ does not result in perfect recovery. The choice of $m_{15}=0.65$ is not particularly unfortunate---averaging over various completions does not improve the metrics. We conclude that while learning useful structure in either case, the Direct Lyapunov Lasso is consistent only for the considered path. Our subsequent analysis explains this behavior, which is a consequence of the failure of the irrepresentability condition in~\eqref{eq:irrcond}.
\end{example}

\begin{figure}[t]
    \centering\def\figmod{.8}
\begin{tabular}{c@{\hspace{1.5cm}} c}
\begin{tikzpicture}[->,every node/.style={circle,draw},line width=1pt, node distance=1.25cm]
%\hspace*{-1cm}
  \node (1) at (0,0)     {$1$};
  \node (2) at (1*\figmod,-2*\figmod)     {$2$};
  \node (3) at (3*\figmod,-2*\figmod)     {$3$};
  \node (4) at (4*\figmod,0)     {$4$};
  \node (5) at (2*\figmod,1*\figmod)     {$5$};
\foreach \from/\to in {1/2,2/3,3/4,4/5}
\draw (\from) -- (\to);   
\end{tikzpicture}
&
\begin{tikzpicture}[->,every node/.style={circle,draw},line width=1pt, node distance=1.25cm]
%\hspace*{-1cm}
  \node (1) at (0,0)     {$1$};
  \node (2) at (1*\figmod,-2*\figmod)     {$2$};
  \node (3) at (3*\figmod,-2*\figmod)     {$3$};
  \node (4) at (4*\figmod,0)     {$4$};
  \node (5) at (2*\figmod,1*\figmod)     {$5$};
\foreach \from/\to in {1/2,2/3,3/4,4/5}
\draw (\from) -- (\to);  
\foreach \from/\to in {5/1}
\draw[color=red] (\from) -- (\to); 
\end{tikzpicture}
\end{tabular}
\caption{Left: The graph $G_1$, a path 1 to 5. \  Right: The graph $G_2$, the 5-cycle.}
\label{tab:intro}
\end{figure}

\begin{figure}[t]
    \centering
    \includegraphics[width=0.99\textwidth]{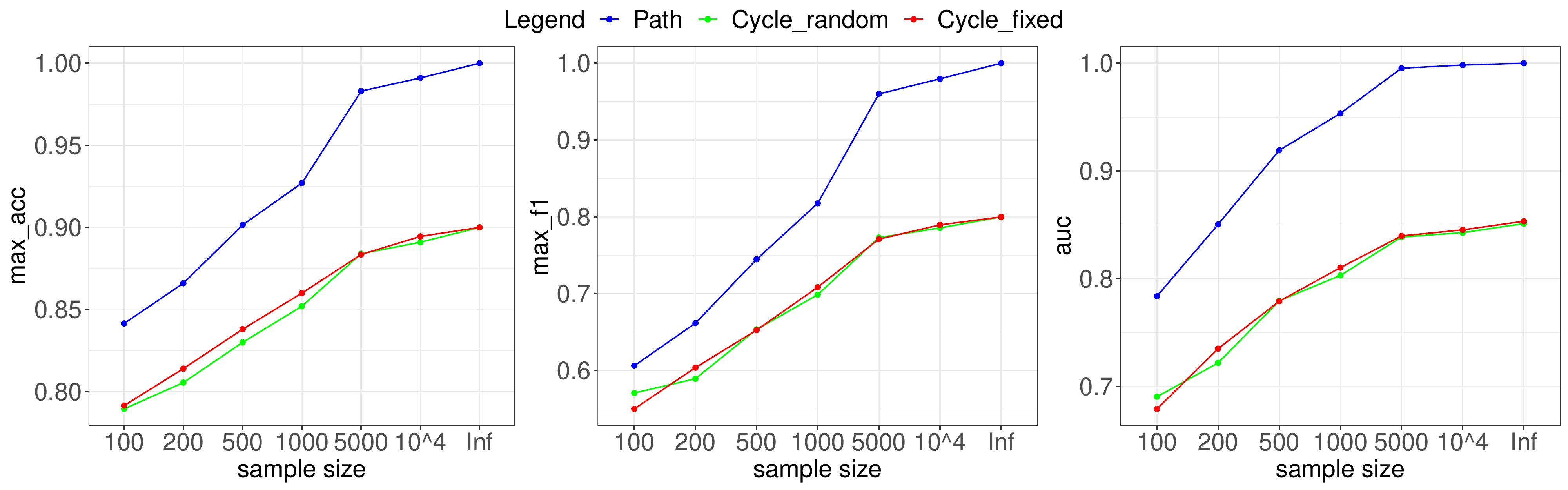}
     \caption{Performance measures for sample sizes $n=10^1,\dots,10^4, \infty$ and models given by the graphs $G_1$  and $G_2$ from Figure~\ref{tab:intro}, one choice of edgeweights for the path, one choice of edgeweights for the cycle (Cycle fixed) and 100 random completions to the 5-cycle (Cycle random). Left: maximal accuracy, Middle: maximal $\text{F}_1$-score, Right: area under the ROC curve.}
    %= path from 1 to 5, $G_2$=  5-cycle}
    \label{fig:SR1}
\end{figure}

\subsection{Notation}

Let $b\in[1,\infty]$.  The $\ell_b$-norm of $v\in\mathbb{R}^p$ is $
\|v\|_{b}=( \sum_{i=1}^{p} |v_{i}|^{b} )^{1/b}
$,
with $\|v\|_{\infty}=\max_{1\leq i \leq n}|v_i|$.
We may apply this vector norm to a matrix $A=(a_{ij})\in \mathbb{R}^{p\times p}$ and obtain the norm $\|A\|_{b}=(\sum_{i=1}^{p} \sum_{j=1}^{n} |a_{ij}|^{b}  )^{1/b}$.  In particular, $\|A\|_{F}:=\|A\|_{2}$ is the Frobenius norm.
We denote the associated operator norm by $\opnorm{A}{b}=\max\{ \|Ax\|_b : \|x\|_b=1\}$. Specifically, we use $\opnorm{A}{2}$ to denote the spectral norm, given by the maximal singular value of $A$, and
$\opnorm{A}{\infty}=\max_{1\leq i \leq p} \sum_{j=1}^{p} |a_{ij}|$ to denote the maximum absolute row sum.

For an index set $S$, we write $A_{\cdot S}$ for the submatrix of $A=(a_{ij})$ that is obtained by selecting the columns indexed by $S$.
The matrices $A_{S \cdot}$ and $A_{SS}$ are defined analogously by selection of rows or both rows and columns, respectively.  
The vec-operator stacks the columns of $A$, giving the vector $\text{vec}(A)= (a_{11}, a_{21}, \dots, a_{p1},\dots,a_{1p},\dots,a_{pp})^{\top}\in\mathbb{R}^{p^2}$.  The diag-operator turns a vector $v\in\mathbb{R}^p$ into the diagonal matrix $\text{diag}(v)\in\mathbb{R}^{p\times p}$ that has $v_i$ as its $i$-th diagonal entry. The Kronecker product of $A$ and another matrix $B=(b_{uv}) \in \RR^{p \times p}$ is denoted by $A\otimes B$.  It is a matrix in $\RR^{p^2 \times p^2}$ with the entries $[A\otimes B]_{(i-1)p+j,(k-1)p+l}=A_{ik}B_{jl}$. The commutation matrix in $\mathbb{R}^{p^2 \times p^2}$ is the permutation matrix $K^{(p,p)}$ such that $K^{(p,p)} \text{vec}(A) = \text{vec}(A^{\top})$. See, e.g., \cite{bernstein2016}.

Finally, we write $\text{Sym}_p$ for the space of symmetric matrices in $\mathbb{R}^{p\times p}$.  We use $\text{PD}_p$ to denote the cone of $p \times p$ positive definite matrices.  The set of stable matrices in $\mathbb{R}^{p\times p}$ is denoted $\text{Stab}_p$.

%%%%%%%%%%%%%%%%%%%%%%%%%%%%%% 

\section{Gram Matrix of the Direct Lyapunov Lasso}
\label{sec:hessian}

In this section, we rewrite the smooth part of the objective of the Direct Lyapunov Lasso from~\eqref{eq:Frobeniuseq} in terms of the vectorized drift matrix and present the resulting Hessian matrix.  

The Lyapunov equation from \eqref{eq:lyapunoveq} is a linear matrix equation and may be rewritten as
\begin{equation}
\label{eq:veclyapunov1}
    A(\Sigma)\,\text{vec}(M)+\text{vec}(C)=0,
\end{equation}
where the $p^2 \times p^2$ matrix $A(\Sigma)$ has its rows and columns indexed by pairs $(i,j)\in\{1,\dots,p\}^2$ and takes the form
\begin{equation}
\label{eq:ASigma}
   A(\Sigma) = (\Sigma \otimes I_{p})+(I_{p} \otimes \Sigma) K^{(p,p)}.
\end{equation}
We have $\text{vec}(M\Sigma) =(\Sigma \otimes I_{p}) \text{vec}(M)$ and $\text{vec}(\Sigma M^{\top}) = (I_{p} \otimes \Sigma)K^{(p,p)} \text{vec}(M)$.  By the symmetry of the Lyapunov equation, $A(\Sigma)$ has two copies of each row corresponding to an off-diagonal entry in the Lyapunov equation. Retaining this redundancy will be helpful for later arguments, as it preserves the Kronecker product structure in \eqref{eq:ASigma}. A display of the matrix $A(\Sigma)$ is provided in Example~\ref{exa:ASigmareduced} in the appendix. Define the \emph{Gram matrix}
\begin{equation}
\label{def:Gamma}
    \Gamma(\Sigma) := A(\Sigma)^{\top} A(\Sigma)\;\in\mathbb{R}^{p^2\times p^2}
\end{equation}
and the vector
\begin{equation}
 \label{def:g}
   g(\Sigma) := -A(\Sigma) \mathrm{vec}(C)\;\in\mathbb{R}^{p^2}.
\end{equation}
Omitting a constant from the objective function, the Direct Lyapunov Lasso problem from \eqref{eq:Frobeniuseq} may be reformulated as
\begin{equation}
\label{eq:lyapunovoptim}
\min_{M \in \RR^{p\times p}}
\frac{1}{2} \, \mathrm{vec}(M)^{\top} \Gamma(\hat{\Sigma}) \mathrm{vec}(M) - g(\hat{\Sigma})^{\top} \mathrm{vec}(M) 
+\lambda \norm{\mathrm{vec}(M)}_{1}. 
\end{equation}

As noted in the introduction, one difficulty that arises in the analysis of the solution of \eqref{eq:lyapunovoptim} is the fact that the Gram matrix has entries that are quadratic polynomials in $\Sigma$ with $p$ terms (i.e., the number of terms scales with the size of the problem). This fact can be seen in the appearance of $\Sigma^2$ in the following formula for the Gram matrix.
\begin{lemma}
\label{lem:gammakronecker}
The Gram matrix for a given covariance matrix $\Sigma$ is equal to
\begin{align*}
    \Gamma(\Sigma) = A(\Sigma)^{\top} A(\Sigma) = 2(\Sigma^2 \otimes I_{p})+ (\Sigma \otimes \Sigma) K^{(p,p)} +K^{(p,p)} (\Sigma \otimes \Sigma). 
\end{align*}
\end{lemma}
\begin{proof}
Apply the rules
$(A \otimes B)^{\top} = (A^{\top} \otimes B^{\top})$, 
$(A \otimes B) (C \otimes D) =(AC) \otimes (BD)$ and $K^{(p,p)} (A\otimes B) K^{(p,p)} = B \otimes A$ to deduce that
\begin{align*}
    A(\Sigma)^{\top}A(\Sigma)&=
   [(\Sigma \otimes I_p)+K^{(p,p)}(I_p \otimes \Sigma)] [(\Sigma \otimes I_p)+(I_p \otimes \Sigma) K^{(p,p)}]\\
    &=2(\Sigma^2 \otimes I_{p})+ (\Sigma \otimes \Sigma) K^{(p,p)} +K^{(p,p)} (\Sigma \otimes \Sigma). 
\end{align*}
\end{proof}

%A proof is given in Appendix~\ref{sec:GramAppendix}.

\section{Consistent Support Recovery with the Direct Lyapunov Lasso} 
\label{sec:bounderror}

We now provide a probabilistic guarantee that the Direct Lyapunov Lasso is able to recover the support of the true population drift matrix that defines the data-generating distribution. 
%We start by introducing more notation. 
Let $M^{*}$ denote the true value of the drift matrix in \eqref{eq:pdimOU}, and let $\Sigma^{*}$ be the associated true covariance matrix of the observations.  We write $\hat{M}$ for the solution of the Direct Lyapunov Lasso problem in \eqref{eq:Frobeniuseq}. The support of $M^*$ is the set of all indices of nonzero elements and is denoted by 
\begin{equation*}
    S\equiv S(M^{*})=\lbrace (j,k):M^{*}_{jk}\neq 0 \rbrace.
\end{equation*}
We write $d=|S|$ for the size of the support of $M^{*}$. The support set of the estimate $\hat M$ is 
\begin{align*}
    \hat{S} \equiv S(\hat{M})=\lbrace (j,k):\hat{M}_{jk}\neq 0 \rbrace.
\end{align*}
Let $\hGamma = \Gamma(\hat \Sigma)$, $\sGamma = \Gamma(\Sigma^*)$, $\hg = g(\hat \Sigma)$, $\sg = g(\Sigma^*)$. Furthermore, let $\DGamma = \hGamma - \sGamma$ and $\Dg = \hg - \sg$, and define the quantities 
\begin{align*}
c_{\Gamma^{*}}=\opnorm{(\sGamma_{SS})^{-1}}{\infty}\, \  \text{and}\,\, \ c_{M^{*}} =    \norm{\text{vec}(M^*)}_{\infty}.
\end{align*}
The definition of $c_{\Gamma^{*}}$ requires $\Gamma^{*}_{SS}$ to be invertible, which is an implicit assumption on the identifiability of the parameters; see Remark~\ref{rem:identifiability} in the Appendix.
By suitably adapting work on  structure learning for undirected graphical models \citep{lin2015estimation}, one can derive a deterministic guarantee for success of the Direct Lyapunov Lasso provided  the estimation errors $\DGamma$ and $\Dg$ are sufficiently small (Theorem~\ref{thm:deterministicres} in the Appendix).  This result leads to 
% succeeds  
% Theorem~\ref{thm:deterministicres} provides a deterministic result for the estimation error and support recovery under a general bound on $\DGamma$ and $\Dg$.  It leads to 
the following probabilistic result.

% The result builds on an adaptation of a deterministic result of~\cite{lin2015estimation}, who considered structure learning for undirected graphical models with score matching. This result can be found in Appendix~\ref{sec:supprec}.

\begin{theorem}
\label{thm:probsupport}
Suppose that the data are generated as $n$ i.i.d.~draws from the Gaussian equilibrium distribution of a $p$-dimensional Ornstein-Uhlenbeck process defined by a drift matrix $M^{*} \in \text{Stab}_p$ and a matrix $C \in \PD_p$.  Let $S$ be the support of $M^{*}$.  Assume that $\Gamma^{*}_{SS}$ is invertible and that the irrepresentability condition 
\begin{equation}
\label{eq:irrcond}
           \opnorm{\Gamma^{*}_{S^cS} (\Gamma^{*}_{SS})^{-1}}{\infty}<1-\alpha
\end{equation}
holds for $\alpha \in (0,1]$. 
Let $c_{\Sigma^{*}}=\opnorm{\Sigma^{*}}{2}$, $c_{C} = \norm{\text{vec}(C)}_{2}$,
\begin{align*}
&\tilde{c}= \max \left\lbrace \frac{4 \max \lbrace 1,c_{\Sigma^{*}}^2 \rbrace (4+8c_{\Sigma^{*}} )^2}{c_3}, 16c_{1}^2c_{\Sigma^{*}}^2 (4+8c_{\Sigma^{*}} )^2, \frac{16 \max \lbrace 1,c_{\Sigma^{*}}^2 \rbrace c_{C}^2}{c_3}, 64c_{1}^2c_{\Sigma^{*}}^2 c_{C}^2  \right \rbrace,\\
&
c_{*} = \frac{6}{\alpha}c_{\Gamma^{*}},
\end{align*}
where $\lbrace c_{i} \rbrace_{j=1}^{3}$ are universal constants (from Theorem~\ref{thm:spectraltheo} below) with $c_{1}>\max \lbrace 1,\opnorm{\Sigma^{*}}{2} \rbrace$. Suppose the sample size satisfies $n > \tau_{1} \tilde{c} d p  \max\lbrace c_{*}^{2}, {1}/{4} \rbrace$ for $\tau_{1} > 1$, and the regularization parameter is chosen as
\begin{align*}
    \lambda > \frac{3c_{M^{*}}(2-\alpha)}{\alpha} \sqrt{\frac{\tau_{1} \tilde{c} d p}{n}}.
\end{align*}
Then the following statements hold with probability at least $1-c_{2} \exp \left( - \tau_{1}p\right)$:
\begin{itemize}
    \item[a)] The minimizer $ \hat{M}$ is unique, has its support included in the true support $(\hat{S}\subset S)$, and satisfies 
    \begin{align*}
        \|\hat{M}-M^{*}\|_{\infty} < \frac{2c_{\Gamma^{*}}}{2-\alpha} \lambda.
    \end{align*}
    \item[b)] Furthermore, if 
    \begin{align*}
        \underset{\underset{(j,k) \in S}{1\leq j <k \leq m}}{\min}|M^{*}_{jk}| > \frac{2c_{\Gamma^{*}}}{2-\alpha}\lambda,
    \end{align*}
    then $\hat{S}=S$ and $\mathrm{sign}(\hat{M}_{jk})=\mathrm{sign}(M^{*}_{jk})$ for all $(j,k) \in S$.
\end{itemize}
\end{theorem}

The reader may be surprised by the sample size requirement  
%Our result requires
%allows the Direct Lyapunov Lasso to succeed in support recovery under 
%a sample size 
of $n = \Omega(d p)$; recall that $d=|S|$ is the size of the support of $M^*$. Since $S$ includes the diagonal of $M^*$, we have $d \geq p$.  Under sparsity, however, $dp$ is not much larger than the number of unknown parameters $p^{2}$, making for a non-trivial guarantee.
%the requirement for the sample size ($n = \Omega(d p)$) almost equal to the number of parameters ($p^2$).  
%However, 
This said, $\Omega(dp)$ is far larger than
%To make a comparison to other (simpler) graphical models, we note that 
the sample size requirement a reader may be familiar with from the glasso for learning undirected conditional independence, which is on the order of $d^2 \log p$ but with $d$ being the maximum number of nonzero entries in any row of a true precision matrix \citep{ravikumar2008highdimensional}.  This allows for far higher-dimensional settings, but 
crucially relies on 
%than our result above and is due to
the glasso having a Hessian that concentrates well entry-wise and a simple connection between the covariance matrix and the sparse precision matrix.  In contrast, the Lyapunov Lasso has a denser Hessian/Gram matrix that includes entries that become heavier-tailed as the dimension $p$ grows.
%and w a more variable sample Hessian and a more %complicated relationship between the covariance %matrix and the sparse signal (i.e., drift %matrix).

In order to prove  Theorem~\ref{thm:probsupport}, we combine the aforementioned deterministic analysis, stated in Theorem~\ref{thm:deterministicres}, with the concentration results we obtain in Section~\ref{sec:concentrationhessian}.  We defer the detailed proof 
%of Theorem~\ref{thm:probsupport} 
to Appendix~\ref{app:probresult}. Here, we present only a short sketch.
\begin{proof}[Sketch]
Applying Theorem~\ref{thm:deterministicres} to obtain the probabilistic result in Theorem~\ref{thm:probsupport} requires showing that a probability of the form $\mathbb{P}(\opnorm{(\Gamma(\hat\Sigma)-\Gamma(\Sigma^*))_{\cdot S}}{\infty}\ge \epsilon_1)$ is small when the sample size is suitably large. Representing the Hessian $\Gamma$ as Kronecker products of the covariance matrix $\Sigma$, Lemma~\ref{lem:gammakronecker} permits to trace back the problem of deriving a concentration inequality for $\Gamma$ to finding concentration inequalities for $\Sigma$. Ultimately, this connection is made in Lemma~\ref{lem:express error in delta} securing that if 
 \begin{align*}
        \opnorm{\DSigma}{2}=\opnorm{\hat{\Sigma}-\Sigma^{*}}2 <  \min \left\lbrace \frac{\epsilon_1}{\sqrt{d}(4+8 c_{\Sigma^{*}})},\frac{\epsilon_2}{2 c_{C}}  \right \rbrace,
\end{align*}
then it holds that
\begin{align*}
\opnorm{(\DGamma)_{\cdot S}}{\infty} < \epsilon_1  
\quad  \text{ and } \quad  
\norm{\Dg}_{\infty} < \epsilon_2.
\end{align*}
A known concentration result for $\Sigma$ is given in Theorem~\ref{thm:spectraltheo}. Combining this with Lemma~\ref{lem:express error in delta} and using the concentration result for $\Gamma$ in Theorem~\ref{thm:deterministicres} yields Theorem~\ref{thm:probsupport}.
\end{proof}

Sample size aside, the crucial assumption for support recovery is the irrepresentability condition.  A detailed analysis of this condition is the subject of the next section.
%Section~\ref{sec:irr}. 

\section{Irrepresentability Condition}
\label{sec:irr}
The irrepresentability condition is vital for Theorem~\ref{thm:probsupport}.
%(and Theorem~\ref{thm:deterministicres}). 
The condition is well known from the standard lasso regression, but turns out
%appears 
to be much more subtle for the Direct Lyapunov Lasso. In regression and in the Lyapunov model, the irrepresentability condition makes an assumption about the Gram matrix in light of the signal. However, the Gram matrix in regression depends solely on the predictors, whereas the Gram matrix for the Lyapunov model is obtained from the matrix $A(\Sigma)$ which depends on the signal itself (Example~\ref{exa:ASigmareduced}). 

We present an analysis of irrepresentability for the Direct Lyapunov Lasso under weak dependence and find that the condition is very restrictive.  Indeed, it leads to a restrictive ordering condition on the diagonal of the drift matrix;  in particular, the support must define a DAG.
% While for every support corresponding to a directed acyclic graph (DAG), there exists a drift matrix that fulfills the irrepresentability condition of Theorem~\ref{thm:probsupport}, this .  However, this requires a restrictive ordering condition on diagonal elements. 
For cyclic graphs, it seems to difficult to construct general examples of irrepresentability.  Simulations suggest that such examples do exist but are rare.  We refer to Appendix~\ref{sec:irrappendix} for details on cyclic graphs as well as a discussion of a weaker notion of irrepresentability that is necessary for consistency.
% For more details on cyclic graphs, the restrictiveness of the condition, and the existence of a weaker notion, we refer to the Appendix~\ref{sec:irrappendix}.

In our study of the irrepresentability condition, we will consider the case where the volatility matrix $C$ is a multiple of the identity; specifically, we assume $C=2I_p$ \emph{throughout this section}. Other diagonal matrices $C$ would also be tractable for analysis and would yield analogous conclusions. Before proceeding, we recall that a matrix $M^{*}\in \text{Stab}_{p}$ with support $S$ satisfies the irrepresentability condition if 
\begin{equation}
\label{eq:irrcondition}
      \rho(M^{*}) :=   \|\Gamma^{*}_{S^{c}S}(\Gamma^{*}_{SS})^{-1}\|_{\infty}
\end{equation}
is strictly smaller than 1;  the condition in~\eqref{eq:irrcond} stated an explicit gap $\alpha>0$.  In the following, we will refer to the number $\rho(M^{*})$ as the \emph{irrepresentability constant} of $M^{*}$.

In standard lasso regression, the irrepresentability condition is fulfilled when each irrelevant predictor exhibits little correlation with the active predictors.  In particular, the condition would hold in a neighborhood of a diagonal Gram matrix.

\begin{figure}
    \centering
    \includegraphics[width=0.85\textwidth]{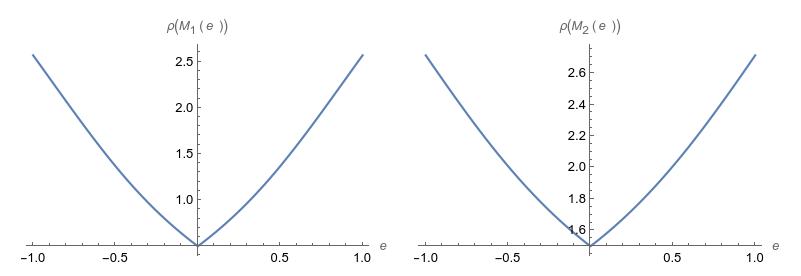}
    \caption{Values of the irrepresentability constants
      $\rho(M_1(e))$ and $\rho(M_2(e))$ for the two matrices from
      Figure~\ref{fig:examplegraph} plotted against the size of the
      off-diagonal entries $e$.  
      Left:   $\rho(M_1(e))$ where $\diag(M_1(e))=(-1/2,-1,-3/2)$.  Right: 
      $\rho(M_2(e))$  where $\diag(M_2(e))=(-3/2,-1,-1/2)$.}
    \label{fig:curveIrr}
\end{figure}

\begin{example}
Consider the graph $G=(V,E)$ in Figure~\ref{fig:examplegraph}, a path on 3 nodes.  For small $e\in\mathbb{R}$, we define two stable matrices $M_1(e)$ and $M_2(e)$ with support given by $G$.  We set their diagonals to $\diag(M_1(e)) = (-1/2,-1,-3/2)$ and $\diag(M_2(e)) = (-3/2,-1,-1/2)$, respectively, and set all nonzero off-diagonal entries equal to $e$.  Note that the diagonal of $M_2(e)$ is the reverse of the diagonal of $M_1(e)$. In Figure \ref{fig:curveIrr}, we plot the two irrepresentability constants $\rho(M_1(e))$ and $\rho(M_2(e))$ as functions of the off-diagonal value $e$.  We observe that irrepresentability holds in a neighborhood of $M_1(0)$, but not around $M_2(0)$. 
\end{example}

In the example, the order of diagonal entries is seen to impact whether irrepresentability holds near a diagonal matrix.  As we prove in the theorem below, this fact is not a coincidence, but rather a general phenomenon. Let $S\subseteq\{(i,j):1\le i,j\le p\}$ be a given support set.  We say that the irrepresentability condition for the support $S$ holds uniformly over a set $U\subset\text{Stab}_p$ if there exists $\alpha>0$ such that $\rho(M^{*})\le 1-\alpha$ for all $M^{*}\in U$ with support $S(M^{*})=S$.  By our convention, the edge set of a directed graph $G=(V,E)$ determines the support set $S_G=\{(j,i):i\to j\in E\}$.
\begin{theorem}
\label{thm:oderingdiagonal}
Let $G=(V,E)$ be a graph with $p$ nodes. Let  $M^0= \text{diag}( -d_1,\dots,-d_p )$ be a stable diagonal matrix. Then, the irrepresentability condition for support $S_G$ holds uniformly over a neighborhood of $M^0$ if and only if
\[
   d_i < d_j \ \ \text{for every edge} \ \ i \to j \in E.
\]
In particular, it is necessary that the graph $G$ is a DAG.
\end{theorem}

We give a proof and an illustrating example in Appendix~\ref{sec:irrproofandexample}. 

%Theorem~\ref{thm:oderingdiagonal} ensures that %there exist drift matrices supported over DAGs %of arbitrary size fulfilling the %irrepresentability condition and making %Theorem~\ref{thm:probsupport} applicable. In %fact, a weaker notion of the irrepresentability %condition is necessary for consistency, as we %show in Appendix~\ref{sec:weakirr}. Examples of %cyclic graphs that fulfill the weaker notion %are provided in %Appendix~\ref{sec:SimStudiesweakvsnormal}.

\section{Simulation Studies}
\label{sec:robust}

In this section, we present simulation studies that provide insight into the performance of the Direct Lyapunov Lasso in seemingly unfavorable settings.  First, most drift matrices do not satisfy the irrepresentability condition;  compare Section~\ref{sec:SimStudiesweakvsnormal} in the Appendix. Second, while our assumption that $C$ is fixed up to a scalar multiple is made similarly in the related case where actual time series data is considered \citep{Gaiffas2019}, it is an assumption that may be overly simple for many applications. 
%However, provided a theory of identifiability, future work should %focus on estimating $(M,C)$ jointly. 
Nevertheless, our simulations suggest robustness of the Direct Lyapunov Lasso to the irrepresentability condition not being fulfilled and to mild misspecification of the volatility matrix $C$, where by robustness we mean that a part of signal is being learned correctly.

For the simulations in this section, we use a similar setting as in \cite{hansen2020}. Each stable matrix $M$ was generated with $M_{ij}=\omega_{ij} \epsilon_{ij}$ for $i \neq j$ and $M_{ii}=-\sum_{j \neq i} |M_{ij}|-|\epsilon_{ii}|$ where $\omega_{ij} \sim \text{Bernoulli}(d)$ and $\epsilon_{ij} \sim N(0,1)$. Unlike in \cite{hansen2020}, we consider four different choices for $C$. The label in brackets corresponds to the one used in Figure \ref{fig:variousC}.
\begin{itemize}
    \item[1)] We choose $C=2I_p$ (C\_ID).
    \item[2)] We choose $C$ with $C_{ii}=u_{ii}$ and $C_{ij}=0$ for $i \neq j$ and $u_{ii} \sim \mathrm{Unif}[0.5,4]$ (C\_Random\_Diag).
    \item[3)] We choose $C$ with $C_{ii}=u_{ii}$ and $C_{ij}=0$ for $i \neq j$ and $u_{ii} \sim \mathrm{Unif}[2,4]$\\ (C\_Random\_Min\_Diag).
    \item[4)] We choose $C_{ij}=\tilde{C}+\tilde{C}^{\top}$ for $i \neq j$ with $\tilde{C}_{ij}=\tilde{\omega}_{ij} \epsilon_{ij}$ for $i \neq j$ and $\tilde{C}_{ii}=0$ with $\tilde{\omega}_{ij} \sim \text{Bernoulli}(2/p)$ and $\epsilon_{ij} \sim N(0,1)$. Finally, we choose $C_{ii}=\sum_{j \neq i} |C_{ij}|+|\epsilon_{ii}|+0.5$ (C\_Random\_Full).
\end{itemize}

For each $k \in \lbrace 1,2,3,4 \rbrace$ and $p= \lbrace 10,15,20,25,30,40,50 \rbrace$,
the edge probability is set as $d=k/p$. For each choice of $C$, we generate 100 pairs of signals $(M,C)$. We generate $N = 1000$ observations from a multivariate Gaussian distribution with covariance matrix solving the Lyapunov equation for $(M,C)$. Note that $p^2>N$ for $p=\lbrace 40,50 \rbrace$ which corresponds to the high-dimensional setting. Then we apply the Direct Lyapunov Lasso with $C=2I_p$ for model selection. The results are calculated along the $\lambda$-grid:
 \begin{align*}
0< \frac{\lambda_{\max}}{10^4}= \lambda_{1} < \dots < \lambda_{100} = \lambda_{\max},
 \end{align*}
with $\lambda_{\max}$ being the minimal $\lambda$-value such that $M$ is diagonal. We use all evaluation metrics for support recovery in Definition~\ref{def:refinedmetrics}. The results are displayed in Figure \ref{fig:variousC}.

\begin{figure}
    \centering
    \includegraphics[width=0.85\textwidth]{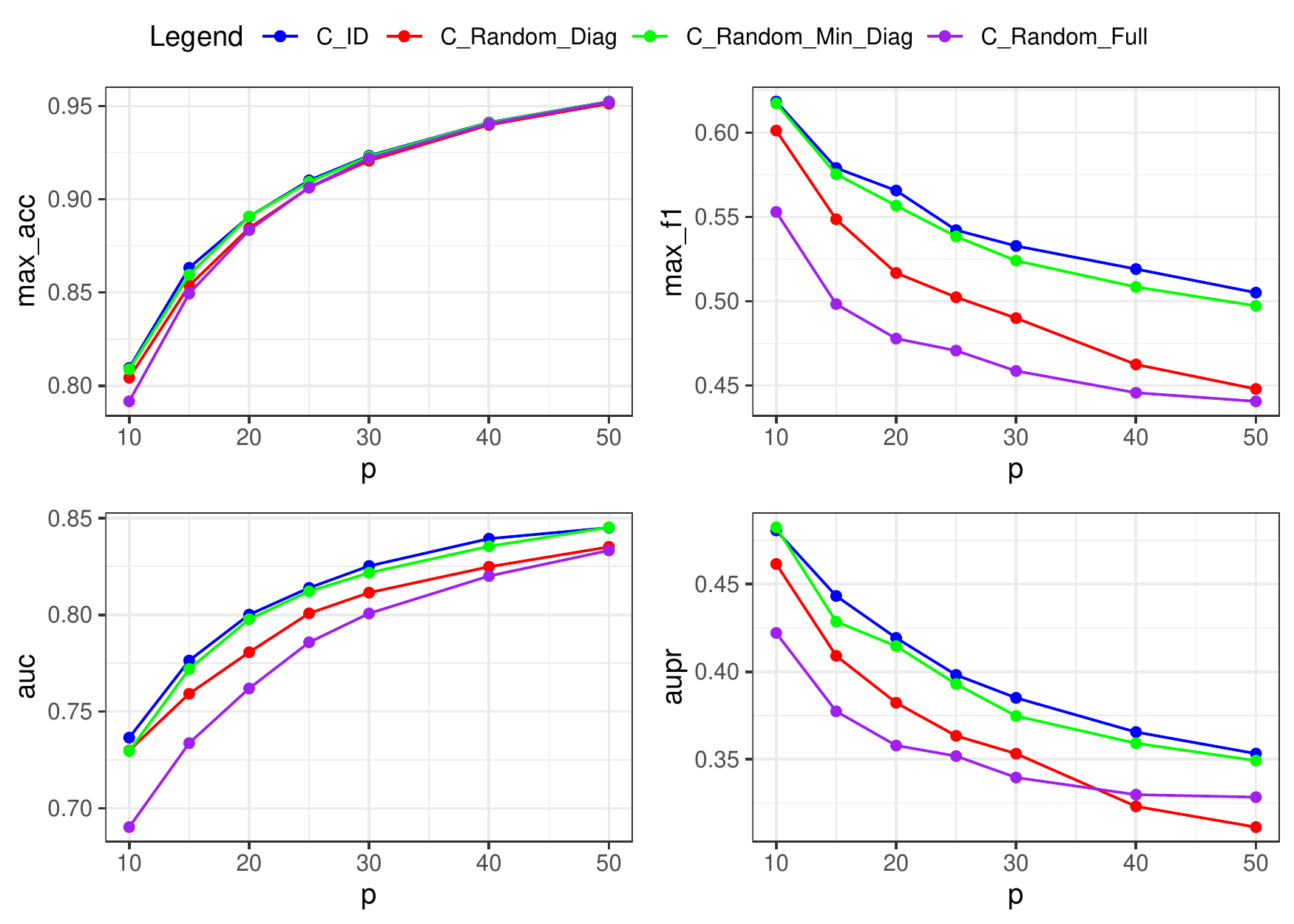}
      \caption{Four evaluation metrics measuring support recovery for the Direct Lyapunov Lasso with $C=2I_p$ where data has been generated using the choices 1)--4) for $C$: 1) is C\_ID, 2) is C\_Random\_Diag, 3) is C\_Random\_Min\_Diag, 4) is C\_Random\_Full.}
      \label{fig:variousC}
\end{figure}

Choice 1) for $C$ is used when applying Direct Lyapunov Lasso for model selection. Therefore, it is natural to assume that this choice should lead to the best results. Choices 2) and 3) allow for variability on the diagonal. The second choice allows for larger differences $(\mathrm{Unif}[0.5,4])$ in the size of the diagonal entries, while the third choice is more conservative $(\mathrm{Unif}[2,4])$. Choices 1) and 3) perform best in our simulations. We observe that there are few differences in all metrics among the choices between choice 1) and choice 3), indicating that the Direct Lyapunov Lasso with $C=2 I_p$ possesses certain robustness to the exact diagonal matrix $C$ of the data generating model. This is true for all $p \in \lbrace 10,15,20,25,30,40,50 \rbrace$. For choice 2), we observe that results in all metrics except maximum accuracy fall with increasing $p$. For $p=40$ and especially for $p=50$, the results for all metrics are similar to choice 4). Choice 4) allows for data generating models for which $C$ is no longer diagonal. For this choice, the worst results are to be expected as the matrix $C$ used for data generation is much different from the one used for estimation.
Another interesting point revealed by the simulations is that although the irrepresentability condition is not satisfied in almost any of the signals, it is still possible to get estimates that recover much of the support of the drift matrix.

\section{Real World Data Example}
\label{sec:realworlddata}

The dataset collected and first analyzed by \cite{Sachs2005} has become a test bed for graphical model selection algorithms. Some recent examples are the review paper of causal discovery methods based on graphical models by \cite{Glymour2019}, the application of classical structure equation models allowing for cycles by \cite{amendola2020} or the application of Lyapunov models and a specific model selection technique by \cite{hansen2020}. The dataset consists of flow cytometry measurements of 11 phosphorylated proteins and phospholipids in human T-cells captured under different experimental conditions, resulting in 14 independent datasets of varying sizes ($n=707$ to $n=927$). When flow cytometry is applied, cells are destroyed during the measurement process and hence the measurements are collected at one point in time. Each sample consists of quantitative and simultaneous measurements of the 11 phosphorylated proteins and phospholipids of single cells. The way in which the dataset was collected matches the motivation for considering
%kes it an interesting choice for
Lyapunov models.
%the parametric assumption.

We apply the Direct Lyapunov Lasso \eqref{eq:Frobeniuseq} with $C=2I_p$ and with an adaptation of the extended Bayesian Information Criterion (BIC) criterion for tuning parameter selection \citep{chen2008,foygel:drton:2010}. First, we standardize every column of the data by calculating $X^{std}_{\cdot,i}= (X_{\cdot,i}-\mu) / \sigma$ where $\mu$ is the mean and $\sigma$ the standard deviation of $X_{\cdot,i}$. We apply the Direct Lyapunov Lasso to obtain estimates along the regularization path that is the logarithmic sequence 
 \begin{align*}
0< \frac{\lambda_{\max}}{10^4}= \lambda_{1} < \dots < \lambda_{100} = \lambda_{\max},
 \end{align*}
where $\lambda_{\max}$ is the minimal $\lambda$-value such that the estimate is diagonal. Extracting the non-zero structure, each estimate $\hat{M}_{j}$ defines a directed graph $G_{j}$ 
%and thus a model $\mathcal{M}_{G_i,2I_{p}}$ 
for $j=1,\dots,100$. To decide which graph to select, we use the extended BIC. We first minimize the two times negative Gaussian log-likelihood 
\begin{align}
L(M)=n \big[\log \det\big(\Sigma(M,2I_p)\big)+\text{tr}\big(\hat{\Sigma}(\Sigma(M,2I_p))^{-1}\big)\big]
\end{align}
for all 100 models obtained by restricting the support of $M$ to $G_j$, $j=1,\dots,100$.
%\mathcal{M}_{G_1,2I_p},\dots,\mathcal{M}_{G_{100},2I_p}$. 
We denote the minima by $\hat{L}_{1},\dots,\hat{L}_{100}$ and substitute these values into
\begin{align}
\label{eq:extendedBIC}
        EBIC_{\gamma}(G_j)= (|E_j|+p) \log n +4 \gamma |E_j| \log p + \hat{L}_j,
\end{align}
where $E_j$ is the edge set of $G_j$.  Selecting the graph with the lowest score yields the Lyapunov model selected by the extended BIC criterion. 

In Figure~\ref{fig:Sachsdata7}a) we present the graph estimate for the protein signaling network using dataset 7. The datasets in \cite{Sachs2005} also contain a graph that shows the conventionally accepted signaling molecule interactions to which we refer as ``ground truth'' knowing about a certain ambiguity for some connections \cite[Section 2]{Ramsey2018FASKWI}. A visualization of the ground truth is given in Figure~\ref{fig:Sachsdata7}b).
%the appendix in Figure~\ref{fig:SachsdataGT}. 
The estimate has 17 edges and 11 connections when counting the 2-cycles only once, while the ground truth has 20 edges with no 2-cycles.  The red edges are the edges that are correctly recovered by the estimate. The orange edges are the edges where the reversed edges are present in the ground truth. The black edges are additional edges that are not present in the ground truth. Among the correctly estimated connections are the direct enzyme-substrate relationships Raf $\to$ Mek and PIP3 $\to$ Pcl$-\gamma$. The connections Pcl$-\gamma$ $\to$ PIP2 and PKA $\to$ Raf are missing in Figure ~\ref{fig:Sachsdata7}, but the Pcl$-\gamma$ $\to$ PIP3 $\to$ PIP2 and Pcl$-\gamma$ $\to$ PIP3 $\to$ PIP2 pathways suggest the presence of these interactions. Only the connection Mek $\to$ Erk is not present at all. Overall, the Direct Lyapunov Lasso with extended BIC is an intuitive and easy-to-implement method that produces a sparse estimate with most edges (or their reverse) present in the ground truth and even some additional edges such as Akt $\to$ Raf can be interpreted as connecting pieces of meaningful pathways.

\begin{figure}
    \centering\def\figmod{.8}
a)
\resizebox{6.75cm}{6.75cm}{%    
\begin{tikzpicture}[->,every node/.style={circle,draw},line width=1pt, node distance=1.25cm]
%\hspace*{-1cm}
  \node (PIP3) at (5,0)     {PIP3};
  \node (PIP2) at (4.206268, 2.703204) {PIP2};
  \node (Plcg) at (2.077075, 4.548160) {Plcg};
  \node (Mek) at (-0.7115742,  4.9491072) {Mek};
  \node (Raf) at (-3.274304,  3.778748) {Raf};
  \node (Jnk) at (-4.797465,  1.408663) {Jnk};
  \node (P38) at (-4.797465, -1.408663) {P38};
  \node (PKC) at (-3.274304, -3.778748) {PKC};
  \node (PKA) at (-0.7115742, -4.9491072) {PKA};
  \node (Akt) at (2.077075, -4.548160) {Akt};
  \node (Erk) at (4.206268, -2.703204) {Erk};
  
  %\node (2) at (1*\figmod,-2*\figmod)     {$2$};
  %\node (3) at (3*\figmod,-2*\figmod)     {$3$};
  %\node (4) at (4*\figmod,0)     {$4$};
  %\node (5) at (2*\figmod,1*\figmod)     {$5$};
\foreach \from/\to in  {Raf/Mek,Mek/Raf,Mek/P38,Plcg/Mek,Plcg/PIP3,PIP3/Plcg,PIP3/PIP2,PIP3/PKA,Erk/Akt,Akt/Raf,Akt/Erk,Akt/PKA,PKA/Akt,PKC/P38,P38/Mek,P38/PKC,Jnk/PKC}
\draw (\from) -- (\to);  
\foreach \from/\to in {Raf/Mek,PIP3/Plcg,PIP3/PIP2,PKA/Akt,PIP3/PKA,PKC/P38}
\draw[color=red] (\from) -- (\to); 
\foreach \from/\to in {Jnk/PKC}
\draw[color=orange] (\from) -- (\to); 
\end{tikzpicture}
}\hfill 
b)
\resizebox{6.75cm}{6.75cm}{%      
\begin{tikzpicture}[->,every node/.style={circle,draw},line width=1pt, node distance=1.25cm]
%\hspace*{-1cm}
  \node (PIP3) at (5,0)     {PIP3};
  \node (PIP2) at (4.206268, 2.703204) {PIP2};
  \node (Plcg) at (2.077075, 4.548160) {Plcg};
  \node (Mek) at (-0.7115742,  4.9491072) {Mek};
  \node (Raf) at (-3.274304,  3.778748) {Raf};
  \node (Jnk) at (-4.797465,  1.408663) {Jnk};
  \node (P38) at (-4.797465, -1.408663) {P38};
  \node (PKC) at (-3.274304, -3.778748) {PKC};
  \node (PKA) at (-0.7115742, -4.9491072) {PKA};
  \node (Akt) at (2.077075, -4.548160) {Akt};
  \node (Erk) at (4.206268, -2.703204) {Erk};
  
  %\node (2) at (1*\figmod,-2*\figmod)     {$2$};
  %\node (3) at (3*\figmod,-2*\figmod)     {$3$};
  %\node (4) at (4*\figmod,0)     {$4$};
  %\node (5) at (2*\figmod,1*\figmod)     {$5$};
\foreach \from/\to in  {PKA/Jnk,PKC/Jnk,PKA/P38,PKC/P38,PKA/Akt,PIP3/Akt,PKA/Erk,Mek/Erk,PKA/Mek,PKC/Mek,Raf/Mek,
PKA/Raf,PKC/Raf,PIP2/PKC,Plcg/PKC,PIP3/PIP2,Plcg/PIP2,PIP3/Plcg,PIP3/PKA,PKA/PIP3}
\draw (\from) -- (\to);  
\end{tikzpicture}
}
\caption{a) Estimated Sachs Network using the Direct Lyapunov Lasso and the EBIC criterion with $\gamma=1$ for scoring (Dataset 7).  b) Ground truth (consensus) network from 
 \cite{Sachs2005}.}
\label{fig:Sachsdata7}
\end{figure}

\section{Conclusion}
\label{sec:conclusion}

We investigated the model selection properties of the Direct Lyapunov Lasso when applied to data distributed according to the graphical continuous Lyapunov model. Although the optimization problem that the Direct Lyapunov Lasso solves is similar to the lasso-penalized linear regression objective, there are several surprising differences. We established a reasonable bound on the sample complexity by carefully investigating the Hessian matrix whose elements are sums of $p$ products of covariances. The irrepresentability condition is more subtle under the Lyapunov model than it is in the linear regression setting. We formulated conditions under which the irrepresentability condition is guaranteed to hold for DAGs based on the topological ordering of the nodes. Despite the irrepresentability condition rarely being fulfilled for randomly selected drift matrices and the problem of misspecification of the volatility matrix when applying the Direct Lyapunov Lasso, we showed that the method is rather robust and is able to detect key features of sparse structures also in seemingly unfavorable settings. Similarly, the combination of Direct Lyapunov Lasso and extended BIC is quite intuitive and easy-to-implement, but still manages to recover important structures of a protein-signaling network purely based on observational data.

\subsection{Acknowledgements}

This project has received funding from the European Research Council (ERC) under the European Union’s Horizon 2020 research and innovation programme (grant agreement No 883818).  Philipp Dettling further acknowledges support from the Hanns-Seidel Foundation.

\bigskip 

\noindent{\bf Authors' addresses:}
\smallskip
\small 

\noindent Philipp Dettling,
Technical University of Munich
\hfill {\tt philipp.dettling@tum.de}

\noindent Mathias Drton, 
Technical University of Munich
\hfill {\tt mathias.drton@tum.de}

\noindent Mladen Kolar,
Booth School of Business \hfill {\tt mladen.kolar@chicagobooth.edu}\\
University of Chicago

\bibliography{arxiv_lyapunov}

%\newpage

\appendix

\newpage

\section{Supplementary Information for Example~\ref{exa:pathcycle}}
\label{sec:appendixsupplementexample}

Below we give the exact choices of stable drift matrices and the estimation procedure for Example~\ref{exa:pathcycle}.\\

Let $G_1$ be the path from 1 to 5, and let $G_2$ be the 5-cycle obtained by adding the edge $5\to 1$; see Figure~\ref{tab:intro}. 
For $G_1$ we define a (well-conditioned) stable matrix $M_1^{*}$ by setting the diagonal to $(-2,-3,-4,-5,-6)$ and the four nonzero subdiagonal entries to $0.65$. For $G_2$, we consider two cases. In the first case, we add the entry $m_{15}=0.65$ to $M_1^{*}$ to obtain the matrix $M_2^{*}$.  We then draw 100 samples of size $n= 100, 200, 500, 1000,5000, 10^4,10^5, \infty$ from $N(0,\Sigma_j^{*})$ for $j=1,2$, where $\Sigma_{j}^{*}$ is the covariance matrix obtained from $M_j^{*}$.  When  $n=\infty$, the population covariance matrices are taken as input to the method. In the second case, we generate 100 stable matrices $M_{2,1}^{*}, \dots, M_{2,100}^{*}$ from $M_1^{*}$ by selecting 100 entries $m_{15}$ according to a uniform distribution on $[0.5, 1]$. Let $\Sigma_{2,1}^{*},\dots,\Sigma_{2,100}^{*}$ be the corresponding equilibrium covariance matrices. In the second case, we generate one sample from $N(0,\Sigma_j^{*})$, $j=1,(2,1),\dots,(2,100)$ for each of the sample sizes given above. Direct Lyapunov Lasso is used for support recovery, with the penalty parameter $\lambda$ chosen on a grid $\lambda_{1}=\lambda_{\max}/10^{4},\dots,\lambda_{100}=\lambda_{\max}$ that is equidistant on the log-scale. The value $\lambda_{\max}$ is the minimal $\lambda$-value such that the estimate is diagonal. To implement the Direct Lyapunov Lasso, we use the \texttt{R} package \texttt{glmnet}, which runs a coordinate descent algorithm for fitting the Lasso, see \cite{glmnet2010}. For each data set, we calculate the maximum accuracy, the maximum $\text{F}_1$-score and the area under the ROC curve. We give the details for the metrics in Definition~\ref{def:refinedmetrics}.

\section{Volatility Matrix and Identifiability}
\label{app:lyapunoveqandID}

In this section, we provide more insight into the assumption that the volatility matrix $C$ is known. Based on this assumption, the question of parameter identifiability is solved for most graphs. This property of a model is necessary to derive consistency results as we do in this work. We give the basic idea of parameter identifiability and provide a reference.
\begin{remark}
\label{rem:CID}
Purely from the covariance matrix $\Sigma$, it is not possible to determine whether the true parameter pair is $(M,C)$ or $(\gamma M, \gamma C)$ with $\gamma >0$, since the Lyapunov equation is scaling invariant. Assuming that $C$ is known is the best we can do when $C$ is unknown up to a multiplicative scalar $\gamma>0$. This accommodates, in particular, the homoscedastic case with $C=\gamma\, I_p$, where $I_p$ denotes the identity matrix. From a practical perspective, the case $C=2I_p$ is the most useful, as it mirrors the equal variance assumption for structural equation models, see \cite{Peters2014}. The assumption that $C$ is known might seem to be restrictive, but it is also made in related work on estimation of the drift matrix for data collected from a single time series \cite{Gaiffas2019}. Moreover, the above-mentioned identifiability theory needs to be further developed to adequately address the case where $C$ is unknown. However, we do think that this should be the subject of further research.
\end{remark}

\begin{remark}
\label{rem:identifiability}
It is evident that the equilibrium distribution of the observations does not uniquely determine the pair of drift and volatility matrices $(M,C)$, as discussed. Specifically, if $(M,C)$ solves the Lyapunov equation, then any scalar multiple of $(M,C)$ also solves the equation. As scaling does not alter the support of the drift matrix $M$, we address this issue by assuming that $C$ is fully known.
Even after reducing to the case of a known volatility matrix $C$, the drift matrix $M$ is not identifiable without exploiting further structure, such as sparsity. Indeed, the Lyapunov equation is a symmetric matrix equation with $(p+1)p/2$ individual equations, whereas $M$ contains $p^2$ unknown parameters. However, $M$ becomes identifiable when it is known to be suitably sparse. For example, it can be shown that the Lyapunov equation for a given $C$ never has two different lower-triangular solutions. In particular, the matrix $M$ from Example~\ref{exa:graphM} can always be uniquely recovered from the Lyapunov equation when its graph/support is known. The work by \cite{dettling2022id} proves that unique recovery of $M$ is always possible for graphs that do not contain cycles of length two. Many graphs with two-cycles permit almost sure unique recovery when the sparse entries of the drift matrix are randomly selected according to a continuous distribution, although here a concise sufficient condition has not been found.
\end{remark}

\section{Display of the Matrix $A(\Sigma)$}
\label{sec:GramAppendix}

In this section, we present a display of the matrix $A(\Sigma)$.
\begin{example}
\label{exa:ASigmareduced}
When $p=3$, the matrix $A(\Sigma)$ is a $9\times 9$ matrix and has the form
\begin{align*}
\begin{blockarray}{cccccccccc}
&(1,1) & (2,1) & (3,1) & (1,2) & (2,2) & (3,2) & (1,3) & (2,3) & (3,3) \\
\begin{block}{c(ccccccccc)}
            (1,1)&2 \Sigma_{11} &0&0&2 \Sigma_{12} & 0 &0 & 2\Sigma_{13} & 0 & 0 \\
            (1,2)&\Sigma_{21} & \Sigma_{11} &0& \Sigma_{22}  & \Sigma_{12} &0& \Sigma_{23} & \Sigma_{13} &0\\
            (1,3)&\Sigma_{31} &0&\Sigma_{11}& \Sigma_{23} & 0 & \Sigma_{12}& \Sigma_{33} & 0 & \Sigma_{13}\\
            \textit{(2,1)}&\Sigma_{21} & \Sigma_{11} &0& \Sigma_{22}  & \Sigma_{12} &0& \Sigma_{23} & \Sigma_{13} &0\\
            (2,2)&0 &2\Sigma_{21}& 0& 0 & 2 \Sigma_{22} &0& 0 & 2\Sigma_{23} & 0 \\
            (2,3)&0 & \Sigma_{31}&\Sigma_{21}& 0 & \Sigma_{23}& \Sigma_{22}   & 0 & \Sigma_{33} & \Sigma_{23}\\
            \textit{(3,1)}&\Sigma_{31} &0&\Sigma_{11}& \Sigma_{23} & 0 & \Sigma_{12}& \Sigma_{33} & 0 & \Sigma_{13}\\
            \textit{(3,2)}&0 & \Sigma_{31}&\Sigma_{21}& 0 & \Sigma_{23}& \Sigma_{22}   & 0 & \Sigma_{33} & \Sigma_{23}\\
            (3,3)&0 & 0&2\Sigma_{31}& 0 & 0&2\Sigma_{23} & 0 & 0 & 2 \Sigma_{33}\\
\end{block}
\end{blockarray}.
\end{align*}
Rows with an italicized index correspond to strictly upper triangular entries in the Lyapunov equation from~\eqref{eq:lyapunoveq}. 
\end{example}

%\begin{proof}[Lemma~\ref{lem:gammakronecker}]
%Apply the rules
%$(A \otimes B)^{\top} = (A^{\top} \otimes B^{\top})$, 
%$(A \otimes B) (C \otimes D) =(AC) \otimes (BD)$ and $K^{(p,p)} (A\otimes B) K^{(p,p)} = B \otimes A$ to deduce that
%\begin{align*}
%    A(\Sigma)^{\top}A(\Sigma)&=
%   [(\Sigma \otimes I_p)+K^{(p,p)}(I_p \otimes \Sigma)] [(\Sigma \otimes I_p)+(I_p \otimes \Sigma) K^{(p,p)}]\\
%    &=2(\Sigma^2 \otimes I_{p})+ (\Sigma \otimes \Sigma) K^{(p,p)} +K^{(p,p)} (\Sigma \otimes \Sigma). 
%\end{align*}
%\end{proof}

\section{Deterministic Result on Support Recovery}
\label{sec:supprec}

In this section, we provide the deterministic result that Theorem~\ref{thm:probsupport} is based on. We adapt Theorem 1 in \cite{lin2015estimation} to arrive at our deterministic result. There are a few differences that we solve, most of the steps are exactly the same, however. The underlying construction is the Primal-Dual-Witness (PDW) method \citep{wainwright2009}.
\begin{theorem}
\label{thm:deterministicres}
Let $M^{*} \in \text{Stab}_p$ be the true drift matrix, and let $S$ be its support. Assume that $\sGamma_{SS}$ is invertible and that the irrepresentability condition 
\begin{equation}
           \opnorm{\Gamma^{*}_{S^cS} (\Gamma^{*}_{SS})^{-1}}{\infty}<1-\alpha
\end{equation}
holds with parameter $\alpha \in (0,1]$. Furthermore, assume that $\hat\Gamma$ is a matrix such that 
\begin{align*}
\opnorm{(\DGamma)_{\cdot S}}{\infty}
<\epsilon_{1}, \qquad \norm{\Delta_g}_{\infty}< \epsilon_{2},
\end{align*}
with $\epsilon_{1} \leq \alpha/(6c_{\Gamma^{*}}).$  If 
\begin{align*}
    \lambda > \frac{3(2-\alpha)}{\alpha} \max \lbrace c_{M^{*}},\epsilon_{1},\epsilon_{2} \rbrace ,
\end{align*}
then the following statements hold:
\begin{itemize}
    \item[a)] The LSGE $\hat{M}$ is unique, has its support included in the true support $(\hat{S} \subseteq S)$, and satisfies 
    \begin{align*}
        ||\hat{M}-M^{*}||_{\infty} < \frac{2c_{\Gamma^{*}}}{2-\alpha} \lambda.
    \end{align*}
    \item[b)] If 
    \begin{align*}
        \underset{\underset{(j,k) \in S}{1\leq j <k \leq m}}{\min}|M^{*}_{jk}|>\frac{2c_{\Gamma^{*}}}{2-\alpha}\lambda,
    \end{align*}
    then $\hat{S}=S$ and $\text{sign}(\hat{M}_{jk})=\text{sign}(M^{*}_{jk})$ for all $(j,k) \in S$.
\end{itemize}
\end{theorem}
\begin{proof}
The proof is very similar to the proof of Theorem 1 in \cite{lin2015estimation}. However, there are a few subtle differences and missing explanations that we add in this proof. For all the calculations that are already carried out in \cite{lin2015estimation}, we refer to the original manuscript for these passages.

%, and we defer details to the Appendix~\ref{sec:proofTheorem31}.

%\section{Proofs for Section~\ref{sec:bounderror}}
%\label{sec:proofTheorem31}

We use the PDW technique to prove the result. The estimate $\hat{M}$ satisfies the KKT conditions 
\begin{equation}
\label{eq:KKT}
    \hat\Gamma\text{vec}(\hat{M})-\hat g + \lambda \hat{z} = 0,
\end{equation}
where $\hat{z} \in \partial\|\text{vec}(\hat{M})\|_{1}$ is an element of the subdifferential of the $\ell_1$-norm, that is, the elements of the vector $\hat{z} \in \mathbb{R}^{p^2}$ satisfy that elements 
\begin{align*}
    \hat{z}_{(i,j)}=\begin{cases} 
     \text{sign}(\text{vec}(\hat{M})_{(i,j)})  &\text{if } \text{vec}(\hat{M})_{(i,j)} \neq 0,\\
    \in [-1,1]   &\text{if } \text{vec}(\hat{M})_{(i,j)}=0.
    \end{cases}
\end{align*}
Here, we index $\hat z$ by pairs $(i,j)$ with $1\le i,j\le p$.
The optimization problem in \eqref{eq:lyapunovoptim} is convex as $\Gamma$ is positive semidefinite by construction, and the KKT conditions are necessary and sufficient for a solution to be optimal for the problem. The PDW technique constructs, in three steps, a primal-dual pair $(\hat M, \hat z)$ that satisfies~\eqref{eq:KKT} and has the support of $\hat M$ contained in $S$.

Since the true signal $M^{*} \in \text{Stab}_{p}$ and $C \in PD_p$, there exists a unique positive definite $\Sigma^{*}$ determined by the continuous Lyapunov equation in \eqref{eq:lyapunoveq}. As a result 
\[\sGamma \vsM - \sg = 0,\]
and we can rewrite the KKT conditions in \eqref{eq:KKT} in the following block form 
\begin{align*}
    &\begin{bmatrix}
        \Gamma_{SS}^{*} & \Gamma_{SS^{c}}^{*}\\
        \Gamma_{S^{c}S}^{*} & \Gamma_{S^{c}S^{c}}^{*}
    \end{bmatrix}
    \begin{bmatrix}
    (\DM)_S \\
    (\DM)_{S^c}
    \end{bmatrix}
    \\&+
    \begin{bmatrix}
    (\DGamma)_{SS} & (\DGamma)_{SS^c} \\
    (\DGamma)_{S^cS} & (\DGamma)_{S^cS^c} 
    \end{bmatrix}
    \begin{bmatrix}
    \vhM_S \\
    \vhM_{S^c}
    \end{bmatrix}
    +\begin{bmatrix}
    (\Dg)_{S}\\
    (\Dg)_{S^{c}}
    \end{bmatrix}+
    \lambda \begin{bmatrix}
    \hat{z}_{S}\\
    \hat{z}_{S^{c}}
    \end{bmatrix}
     =\begin{bmatrix}
    0\\
    0
    \end{bmatrix},
\end{align*}
where $\DM = \vhM - \vsM$. We now construct a pair $(\hat M, \hat z)$ that satisfies the equation.

{\textbf{ Step 1.}} We solve the restricted optimization problem
\begin{equation}
    \label{eq:supportrestricted}
    \vtM = \arg\min_{\vM_{S^c} = 0}\ 
    \frac{1}{2} \vM^\top  \hGamma \vM - \hg^{\top} \vM +\lambda \norm{\vM}_{1}.
\end{equation}
Since $\Gamma_{SS}^{*}$ is invertible,  under our assumptions, $\hGamma_{SS}$ is also invertible. The matrix $\hGamma_{SS}$ can be expressed as 
\begin{align*}
\hGamma_{SS}= \Gamma_{SS}^{*} + (\hGamma_{SS}-\Gamma_{SS}^{*}) =\Gamma_{SS}^{*} + (\Delta_{\Gamma})_{SS}
\end{align*}
Factoring out $\Gamma_{SS}^{*}$, we obtain 
\begin{align*}
\Gamma_{SS}^{*} + (\Delta_{\Gamma})_{SS} =\Gamma_{SS}^{*} (I_{|S|} + (\Gamma_{SS}^{*})^{-1}(\Delta_{\Gamma})_{SS})
\end{align*}
where $I_{|S|}$ denotes the identity matrix of size $|S| \times |S|$. Then, the matrix $\hGamma_{SS}$ is invertible if 
\begin{align*}
    \rho((\Gamma_{SS}^{*})^{-1}(\Delta_{\Gamma})_{SS})<1.
\end{align*}
This is true as the spectral norm is bounded by the maximum absolute row sum norm and 
\begin{align*}
\opnorm{(\Gamma_{SS}^{*})^{-1}(\Delta_{\Gamma})_{SS}}{\infty} \leq \opnorm{(\Gamma_{SS}^{*})^{-1}}{\infty} \opnorm{(\Delta_{\Gamma})_{SS}}{\infty} <1
\end{align*}
with the second inequality being true because of $\opnorm{(\Delta_{\Gamma})_{SS}}{\infty}< \epsilon_{1} <  \alpha/6c_{\Gamma^{*}} < 1/c_{\Gamma^{*}}$ and $ c_{\Gamma^{*}}=\opnorm{(\sGamma_{SS})^{-1}}{\infty}$.
%we can follow the derivation in equation (7.8) of \cite{lin2015estimation} to show that,

Therefore, the solution $\vtM$ is unique. Furthermore, we have 
\[
(\vtM)_S = (\hGamma_{SS})^{-1} (\hg_S - \lambda \sign((\vtM)_S).
\]
Let $\tDM = \vtM - \vsM$. Following the proof of Theorem 1 in \citet{lin2015estimation}, we have
\begin{align}
\label{eq:Mdistance}
    \norm{\tDM}_{\infty} \leq \frac{c_{\Gamma^{*}}}{1-\alpha/6} \cdot \frac{6-\alpha}{3(2-\alpha)} \lambda =\frac{2c_{\Gamma^{*}}}{2-\alpha} \lambda.
\end{align}

{\textbf{ Step 2.}} Let $\tilde{z}_{S} = \sign(\vtM_S)$. Then $\tilde{z}_S \in \partial \norm{\vtM}_1$.

{\textbf{ Step 3.}} Let
\begin{multline}
    \label{eq:zsc}
    \tilde{z}_{S^{c}}=
    \frac{1}{\lambda}\left[ 
    -\Gamma_{S^{c}S}^{*}(\Gamma_{SS}^{*})^{-1}((\DGamma)_{SS}\vtM_S+(\Dg)_S)
    + (\DGamma)_{S^{c}S}\vtM_S \right.\\
       \left. + (\Dg)_{S^{c}}+\lambda
        \Gamma^{*}_{S^{c}S}(\Gamma^{*}_{SS})^{-1} \sign(\vtM_S)\right].
\end{multline}
We show that $\norm{\tilde{z}_{S^c}}_1 < 1$, which is a dual feasibility condition. Once this is shown, we have that the pair $(\vtM, \tilde z)$ satisfies \eqref{eq:KKT} by construction, and $(\vhM, \hat z) = (\vtM, \tilde z)$ is the solution to the optimization problem in \eqref{eq:lyapunovoptim}. Furthermore, Lemma 1 of \citet{wainwright2009} implies that the strict dual feasibility implies that $\hat S \subseteq S$. Following Theorem 1 in \cite{lin2015estimation}, we have
\begin{align*}
    \|\tilde{z}_{S^{c}}\|_{\infty} 
    \leq& 
    \underbrace{
    \frac{2-\alpha}{\lambda}\|(\Delta_{\Gamma})_{\cdot S}\text{vec}(M^{*})_{S}\|_{\infty}
    }_{G_1}
    +
    \underbrace{
    \frac{2-\alpha}{\lambda}\opnorm{(\Delta_{\Gamma})_{\cdot S}}{\infty}\|\Delta_{S}\|_{\infty}
    }_{G_{2}}\\
    &+
    \underbrace{
    \frac{2-\alpha}{\lambda}\|\Delta_{g}\|_{\infty}
    }_{G_{3}}+(1-\alpha).
\end{align*}
For $G_1$, we have that
\begin{align*}
    G_{1} 
    \leq \frac{2-\alpha}{\lambda} \opnorm{(\Delta_{\Gamma})_{\cdot S}}{\infty} \|\text{vec}(M^{*})\|_{\infty}
    = \frac{2-\alpha}{\lambda} c_{M^{*}}\epsilon_{1}
    \leq \frac{\alpha}{3}.
\end{align*}
For $G_3$, we have that
\begin{equation*}
    G_{3}=\frac{2-\alpha}{\lambda}\|\Delta_{g}\|_{\infty}<\frac{2-\alpha}{\lambda} \epsilon_{2}\leq \frac{\alpha}{3}.
\end{equation*}
Finally, for $G_{2}$, we have that
\begin{align*}
    G_2<\frac{2-\alpha}{\lambda}\cdot \frac{\alpha}{6c_{\Gamma*}}\cdot \epsilon_{1}  \cdot \frac{c_{\Gamma^{*}}}{1-\alpha/6} \cdot \frac{6-\alpha}{3(2-\alpha)}<\frac{\alpha}{3}.
\end{align*}
Combining these bounds, we have that
$\|\tilde{z}_{S^{c}}\|_{\infty}<1$, which establishes the strict dual feasibility.

Finally, for any $(j,k) \in S$, we have that
\[
|\hat M_{jk}| 
\geq
|M^*_{jk}| - |\hat M_{jk} - M^*_{jk}| 
>
\underset{\underset{(j,k) \in S}{1\leq j <k \leq p}}{\min}|M^{*}_{jk}| - \norm{\vhM - \vsM}_{\infty} > 0 ,
\]
which shows that $\hat S = S$. 
\end{proof}

%We note that the condition 
%\begin{align*}
%    \opnorm{\Gamma^{*}_{S^cS} (\Gamma^{*}_{SS})^{-1}}{\infty}<1-\alpha
%\end{align*}
%is sufficient to prove Theorem~\ref{thm:deterministicres}. In particular, a weaker version holds if \eqref{eq:minimalsignal} is not fulfilled. Then, $\hat{M}$ remains unique and has its support included in the true support $(\hat{S} \subseteq S)$.
%However, the formulation in~\eqref{eq:irrcond} allows one to prove that irrepresentability is also an almost necessary condition, similar to~\cite{Zhao2006}.

%\begin{remark}
%\label{rem:identifiability}
%Identifiability coincides with the injectivity of the map
%\begin{align*}
%    \phi_{G,C}: \text{Stab}_p(E) &\to \mathcal{M}_{G,C}\\
%     M &\to \Sigma(M,C) 
%\end{align*}
%for a graph $G=(V,E)$ and $C \in \PD_p$. The matrix $\Sigma(M,C)$ denotes the solution of %the continuous Lyapunov equation (\ref{eq:lyapunoveq}). This means that within the %statistical model $M_{G,C}$ one is able to recover the true parameters $M \in %\text{Stab}_{p}(E)$ given a fixed $C \in PD_{p}$ based on data. 
%A detailed analysis of identifiability is beyond the scope of this work, but a subject for %further work in these models. For the well-known structural equation models (SEMs), %\cite{Drton2011} addressed this question. Currently, we are working on this topic and %conjecture global identifiability for all simple graphs which does not apply in the case of %SEMs.
%\end{remark}

\section{Probabilistic Analysis}
\label{sec:concentrationhessian}

The Direct Lyapunov Lasso depends on the loss being sufficiently close to its population version in the sense of $\DGamma = \hGamma - \sGamma$ and $\Dg = \hg - \sg$ being sufficiently small.  In this section, we bound $\DGamma$ and $\Dg$ in terms of $\DSigma = \hSigma - \sSigma$ and, subsequently, use a concentration inequality for  $\opnorm{\DSigma}{2}$ to probabilistically bound  $\DGamma$ and $\Dg$.

Deriving an inequality for $\hGamma$ is most critical as the matrix contains sums of products of covariances and a careful analysis is required to obtain a non-trivial requirement on the sample size. Let $\Gamma(\Sigma) = \Gamma_1(\Sigma) + \Gamma_2(\Sigma)$, where
\begin{equation*}
    \Gamma_1(\Sigma) = 2 (\Sigma^{2} \otimes I_p) 
    \quad\text{and}\quad
    \Gamma_2(\Sigma) = (\Sigma \otimes \Sigma)K^{(p,p)}+K^{(p,p)} (\Sigma \otimes \Sigma).
\end{equation*}

\begin{lemma}
\label{lem:concentrationgamma1}
Let $c_{\Sigma^{*}}=\opnorm{\Sigma^{*}}{2}$. Then
\begin{align*}
\opnorm{\Gamma_1(\hSigma) - \Gamma_1(\sSigma)}{2}
\leq 2 \opnorm{\DSigma}{2}^2 +  4 c_{\Sigma^{*}}  \opnorm{\DSigma}{2}.
\end{align*}
\end{lemma}
\begin{proof}
Using that $\opnorm{A \otimes B}{2} = \opnorm{A}{2}\opnorm{B}{2}$, we obtain that
\begin{align*}
     \opnorm{\Gamma_{1}(\hat{\Sigma})-\Gamma_{1}(\Sigma^{*})}{2}
     &= 2\opnorm{(\hat{\Sigma}^{2}-(\Sigma^{*})^2)\otimes I_p)}{2}\\ 
     &= 2\opnorm{ \hat{\Sigma}^{2}-(\Sigma^{*})^2)}{2}\\
     &\leq 2\opnorm{ \DSigma }{2}^2 + 2\opnorm{ \DSigma \Sigma^{*}}{2} + 2\opnorm{ \Sigma^{*} \DSigma}{2}.
\end{align*}
Since the spectral norm of a symmetric matrix is the absolute maximal eigenvalue, and the eigenvalues of a squared matrix are the squared eigenvalues of the original matrix, we find as claimed that
\begin{equation*}
    \opnorm{\Gamma_{1}(\hat{\Sigma})-\Gamma_{1}(\Sigma^{*})}{2} \leq 2\opnorm{\DSigma}{2}^2 +  4 \opnorm{\Sigma^{*}}{2} \opnorm{\DSigma}{2}.
\end{equation*}
\end{proof}

\begin{lemma}
\label{lem:concentrationgamma2}
Let $c_{\Sigma^{*}}=\opnorm{\Sigma^{*}}{2}$. Then
\begin{align*}
    \opnorm{\Gamma_2(\hSigma) - \Gamma_2(\sSigma)}{2}  \leq 2 \opnorm{\DSigma}{2}^2 + 4 c_{\Sigma^{*}}\opnorm{\DSigma}2.
\end{align*}
\end{lemma}
\begin{proof}
The commutation matrix $K^{(p,p)}$ is an orthonormal matrix. Therefore, $\opnorm{K^{(p,p)}}{2}=1$ and 
\begin{align*}
\label{eq:proof:lem:concentrationgamma2:1}
    \opnorm{K^{(p,p)} (\hat{\Sigma} \otimes \hat{\Sigma}-\Sigma^{*} \otimes \Sigma^{*}) }{2} &=  \opnorm{ (\hat{\Sigma} \otimes \hat{\Sigma}-\Sigma^{*} \otimes \Sigma^{*}) K^{(p,p)} }{2} 
    = \opnorm{ \hat{\Sigma} \otimes \hat{\Sigma}-\Sigma^{*} \otimes \Sigma^{*} }{2}.
\end{align*}
We obtain that
\begin{align*}
  \opnorm{\Gamma_2(\hSigma) - \Gamma_2(\sSigma)}{2}
  % &=2\opnorm{ (\hat{\Sigma} \otimes \hat{\Sigma})K^{(p,p)}+K^{(p,p)} (\hat{\Sigma} \otimes \hat{\Sigma})-(\Sigma^{*} \otimes \Sigma^{*})K^{(p,p)}+K^{(p,p)} (\Sigma^{*} \otimes \Sigma^{*})}2\\
  \leq & 2 \opnorm{\hat{\Sigma} \otimes \hat{\Sigma} - \Sigma^{*} \otimes \Sigma^{*}}2 
  %&& \text{(using \eqref{eq:proof:lem:concentrationgamma2:1})}
  \\
  %&=4 \opnorm{(\DSigma+\Sigma^{*})\otimes (\DSigma +\Sigma^{*})-\Sigma^{*} \otimes \Sigma^{*}}2\\
  =&2 \opnorm{\DSigma \otimes \DSigma + \DSigma \otimes \Sigma^{*} + \Sigma^{*} \otimes \DSigma + \Sigma^{*} \otimes \Sigma^{*}-\Sigma^{*} \otimes \Sigma^{*}}2\\
  \leq & 2\opnorm{\DSigma \otimes \DSigma}{2} 
  +  2\opnorm{\DSigma \otimes \Sigma^{*}}{2} + 2\opnorm{\Sigma^{*} \otimes \DSigma}{2}\\
  \leq & 2\opnorm{\DSigma}{2}^2+ 4 \opnorm{\Sigma^{*}}{2} \opnorm{\DSigma}{2},
  %&= 4\opnorm{\DSigma}2^2+8 c_{\Sigma^{*}} \opnorm{\DSigma}2,
\end{align*}
which was the claim.
%which completes the proof.
\end{proof}

For a matrix $A \in \RR^{p \times d}$, it holds that $\opnorm{A}{\infty} \leq \sqrt{d}\opnorm{A}{2}$. Then it follows from Lemma~\ref{lem:concentrationgamma1} and Lemma~\ref{lem:concentrationgamma2} that
\begin{equation}
    \label{eq:bound:DGamma}
    \opnorm{(\DGamma)_{\cdot S}}{\infty}
    \leq 
    \sqrt{d} \rbr{
    4 \opnorm{\DSigma}{2}^2 + 8 c_{\Sigma^{*}}\opnorm{\DSigma}{2}
    }.
\end{equation}
We note that bounding $\opnorm{(\DGamma)_{\cdot S}}{\infty}$ using $\norm{(\DGamma)_{\cdot S}}_{\infty}$, as was done in \cite{lin2015estimation}, leads to a worse bound. While such an approach might seem simpler, it does not exploit the structure of the Hessian $\Gamma$ in Lemma \ref{lem:gammakronecker}.

We now provide a bound on $\norm{\Dg}_{\infty}$. 
\begin{lemma}
\label{lem:concentrationg}
We have $\norm{\Dg}_{\infty} \leq 2c_{C} \opnorm{\DSigma}{2}$, where $c_{C} = \norm{\text{vec}(C)}_{2}$. 
\end{lemma}
\begin{proof}
Similar to the proof of Lemma~\ref{lem:concentrationgamma1} and Lemma~\ref{lem:concentrationgamma2}, we have
\begin{align*}
    \norm{\Dg}_{\infty} 
    &\leq\norm{\Dg}_{2}\\
    %&=\|2\text{vec}(C)^{\top} A(\hat{\Sigma})-2 \text{vec}(C) A(\Sigma^{*})}{2}\\
    &\leq  c_{C} \opnorm{\Sigma^{*} \otimes I_p-(I_p \otimes \Sigma^{*})K^{(p,p)}-\hat{\Sigma} \otimes I_p + (I_p \otimes \hat{\Sigma}) K^{(p,p)}}{2}\\
    % &\leq 2 c_{C} (\opnorm{\hat{\Sigma} \otimes I_p - \Sigma^{*} \otimes I_p}{2}+\opnorm{(I_p \otimes \hat{\Sigma}-I_p \otimes \Sigma^{*})K^{(p,p)}}{2})\\
    &\leq c_{C}(\opnorm{I_p \otimes (\hat{\Sigma}-\Sigma^{*})}{2}+\opnorm{(\hat{\Sigma}-\Sigma^{*})\otimes I_p}{2}) \quad \text{(since $\opnorm{K^{(p,p)}}{2}=1$)}\\
    &=2 c_{C} \opnorm{\DSigma}{2}.
\end{align*}
\end{proof}

The bounds in \eqref{eq:bound:DGamma} and Lemma~\ref{lem:concentrationg} depend on the spectral norm of $\DSigma$. We adapt Theorem 6.5 in \cite{wainwright2009b} to our setting to upper bound $\opnorm{\DSigma}{2}$ under the assumption that $(x_i)_{i=1}^{n}$ are sub-Gaussian. 
\begin{theorem}[Theorem 6.5. in \cite{wainwright2009b}]
\label{thm:spectraltheo}
Suppose that $(X_i)_{i=1}^{n}$ are $\sigma$ sub-Gaussian random variables. 
Then the sample covariance matrix $\hat{\Sigma}$ in \eqref{eq:samplecov} satisfies 
\begin{align*}
    \mathbb{P}\left(\frac{\opnorm{\hat{\Sigma}-\Sigma^{*}}{2}}{\sigma^2} \geq c_1\left\lbrace \sqrt{\frac{p}{n}}+\frac{p}{n}\right\rbrace+\delta \right) \leq c_2 \exp(-c_3 n \min \lbrace \delta,\delta^2 \rbrace) \quad \forall \delta \geq 0,
\end{align*}
where $\lbrace c_j \rbrace_{j=0}^{3}$ are universal constants.
\end{theorem}

\begin{corollary}
\label{cor:spectralcor}
Let $\lbrace c_j \rbrace_{j=1}^{3}$ be the universal constants from Theorem~\ref{thm:spectraltheo}, but ensuring that $c_{1}>\max \lbrace 1,1/\opnorm{\Sigma^{*}}{2}\rbrace$. Let $(X_i)_{i=1}^{n}$ be Gaussian random variables. For any
$\epsilon \in (4c_1 \opnorm{\Sigma^{*}}{2} \sqrt{p/n}, 2)$, we have
\begin{align*}
    \mathbb{P}\left(\opnorm{\hat{\Sigma}-\Sigma^{*} }{2} \geq \epsilon\right) \leq c_2 \exp\left(- \frac{c_3}{4\max(1,\opnorm{\Sigma^{*}}{2}^2)}n\epsilon^{2} \right).
\end{align*}
\end{corollary}
\begin{proof}
A Gaussian random vector is sub-Gaussian with parameter $\sigma = \opnorm{\Sigma^{*}}{2}$. \\
Set $\delta=\min \left(\frac{\epsilon}{2 \opnorm{\Sigma^{*}}{2}}, \frac{\epsilon}{2}\right)$. Since $\frac{p}{n}<\frac{\epsilon^2}{16 c_{1}^2 \opnorm{\Sigma^{*}}{2}^2}$, we have
\begin{align*}
    &\opnorm{\Sigma^{*}}{2} \left( c_{1}  \left \lbrace  \sqrt{\frac{p}{n}}+\frac{p}{n}  \right \rbrace +\delta \right) < c_{1} \opnorm{\Sigma^{*}}{2} \left\lbrace \frac{\epsilon}{4 c_{1} \opnorm{\Sigma^{*}}{2}} + \frac{\epsilon^2}{16 c_{1}^2 \opnorm{\Sigma^{*}}{2}^2} \right\rbrace +\frac{\epsilon}{2}\\
    = &\frac{\epsilon}{4}+ \frac{\epsilon^2}{16 c_{1}\opnorm{\Sigma^{*}}{2}} + \frac{\epsilon}{2} <  \frac{\epsilon}{4} + \frac{\epsilon}{4}+ \frac{\epsilon}{2}  =\epsilon.
\end{align*}
Since $\delta < 1$, it holds that $\delta^{2}<\delta$. Then
\begin{align*}
    \mathbb{P}\left(\opnorm{\hat{\Sigma}-\Sigma^{*} }{2} \geq \epsilon \right) & \leq \mathbb{P}\left(\opnorm{\hat{\Sigma}-\Sigma^{*}}{2} \geq \opnorm{\Sigma^{*}}{2} \left( c_1  \left\lbrace \sqrt{\frac{p}{n}}+\frac{p}{n}\right\rbrace+\delta \right) \right)\\ 
    &    \leq c_{2} \exp(-c_3 n \delta^2)
    = c_2 \exp\left(- \frac{c_3}{4\max(1,\opnorm{\Sigma^{*}}{2}^2)}n\epsilon^{2} \right). 
\end{align*}
\end{proof}

%\begin{proof}
%A Gaussian random vector is sub-Gaussian with parameter $\sigma = 1$. Set $\delta=\frac{\epsilon}{2}$. Since $\frac{p}{n}<\frac{\epsilon^2}{16 c_{1}^2}$, we have
%\begin{align*}
%    c_{1}\left \lbrace  \sqrt{\frac{p}{n}}+\frac{p}{n}  \right \rbrace +\delta < c_{1} \left\lbrace \frac{\epsilon}{4 c_{1}} + \frac{\epsilon^2}{16 c_{1}^2} \right\rbrace +\frac{\epsilon}{2} = \frac{\epsilon}{4}+ \frac{\epsilon^2}{16 c_{1}} + \frac{\epsilon}{2} <  \frac{\epsilon}{4} + \frac{\epsilon}{4}+ \frac{\epsilon}{2}  =\epsilon.
%\end{align*}
%Since $\delta < 1$, it holds that $\delta^{2}<\delta$. Then
%\begin{align*}
%    \mathbb{P}\left(\opnorm{\hat{\Sigma}-\Sigma^{*} }{2} \geq \epsilon \right) & \leq \mathbb{P}\left(\opnorm{\hat{\Sigma}-\Sigma^{*}}{2} \geq c_1\left\lbrace \sqrt{\frac{p}{n}}+\frac{p}{n}\right\rbrace+\delta \right)\\ 
%    &    \leq c_{2} \exp(-c_3 n \delta^2)
%    = c_2 \exp\left(-\frac{c_3}{4}n\epsilon^{2}\right). 
%        \qedhere
%\end{align*}
%\end{proof}

We finally have the following result.
\begin{lemma}
\label{lem:express error in delta}
In the event that
 \begin{align*}
        \opnorm{\DSigma}{2}=\opnorm{\hat{\Sigma}-\Sigma^{*}}2 <  \min \left\lbrace \frac{\epsilon_1}{\sqrt{d}(4+8 c_{\Sigma^{*}})},\frac{\epsilon_2}{2 c_{C}}  \right \rbrace
\end{align*}
it holds that
\begin{align*}
\opnorm{(\DGamma)_{\cdot S}}{\infty} < \epsilon_1  
\quad  \text{ and } \quad  
\norm{\Dg}_{\infty} < \epsilon_2.
\end{align*}
\end{lemma}
\begin{proof}
The result follows directly from \eqref{eq:bound:DGamma}, where $\opnorm{\DSigma}{2}^2\le \opnorm{\DSigma}{2}$, and Lemma~\ref{lem:concentrationg}.
% Using the decomposition of $\Gamma$ introduced in section \ref{sec:concentrationhessian} and that $\opnorm{\delta}{2}<1$ we obtain
% \begin{align*}
%     \opnorm{R_{1,\cdot S}}{\infty} &\underset{\ref{lem:decomposeWY}}{\leq} \sqrt{d} (\opnorm{\hat{W}-W^{*}}2+\opnorm{\hat{Y}-Y^{*}}2)\\
%     &\underset{\substack{\ref{lem:concentrationgamma1}\\ \ref{lem:concentrationgamma2}}}{\leq} \sqrt{d}(4\opnorm{\delta}2^2+8c_{\Sigma^{*}}\|\delta\|_2+4\opnorm{\delta}2^2+8c_{\Sigma^{*}}\opnorm{\delta}2)\\
%     &< \sqrt{d} (8+16 c_{\Sigma^{*}}) \opnorm{\delta}{2}.
% \end{align*}
% Therefore, the condition $\opnorm{R_{1,\cdot S}}{\infty}<\epsilon_1$ is fulfilled if 
% \begin{align*}
%     &\sqrt{d} (8+16 c_{\Sigma^{*}}) \opnorm{\delta}2 < \epsilon_1 \\
%     \Leftrightarrow &\opnorm{\delta}2 < \frac{\epsilon_{1}}{\sqrt{d} (8+16 c_{\Sigma^{*}})}
% \end{align*}
% Using Lemma \ref{lem:concentrationg}, we obtain 
% \begin{align*}
%     \|r_2\|_{\infty} <4 c_{C} \opnorm{\delta}2.
% \end{align*}
% Therefore, the condition $\|r_{2\|_{2}}<\epsilon_{2}$ is fulfilled if 
% \begin{align*}
%     &4 c_{C} \opnorm{\delta}2 < \epsilon_2\\
%     \Leftrightarrow &\opnorm{\delta}2 < \frac{\epsilon_2}{4 c_{C}}.
% \end{align*}
\end{proof}

\section{Proof Probabilistic Guarantee on Support Recovery}
\label{app:probresult}

Using the preparation in Appendix~\ref{sec:concentrationhessian}, we prove the main result. 
\vspace{0.5cm}
\begin{proof}[Proof of Theorem~\ref{thm:probsupport}]
%Let $\epsilon=\sqrt{{\tilde{c} d (\log p^{\tau_{1}})}/{n}}$. Under the assumption on the sample size, we have 
%$\epsilon \in (4c_1\opnorm{\Sigma^{*}}{2} \sqrt{p/n},2)$.  
%Corollary~\ref{cor:spectralcor} then gives us that
%\begin{align*}
%  \opnorm{\Delta_{\Sigma}}{2} <  \min \left\lbrace \frac{\epsilon}{\sqrt{d}(4+8 c_{\Sigma^{*}})},\frac{\epsilon}{2 c_{C}} \right \rbrace
%\end{align*}
%with probability at least $1 - c_2p^{-\tau_1}$.
%Then $\opnorm{(\Delta_{\Gamma})_{\cdot S}}{\infty} < \epsilon$ and $\|\Delta_g\|_{\infty} < \epsilon$,
%using Lemma~\ref{lem:express error in delta}.
%Furthermore, under the assumption on the sample size, we have $\epsilon \leq \frac{\alpha}{6c_{\Gamma^{*}}}$. The result follows from Theorem \ref{thm:deterministicres}.
We prove the result in three steps.
\begin{itemize}
   \item[1)] It has to hold that 
   \begin{align*}
       \frac{\epsilon}{\sqrt{d}(4+8 c_{\Sigma^{*}})}, \frac{\epsilon}{2 c_{C}} \in \left(4 c_1 \opnorm{\Sigma^{*}}{2} \sqrt{p/n},2\right).
   \end{align*}
   \item[2)] Then Corollary~\ref{cor:spectralcor} gives us that
   \begin{align*}
  \opnorm{\Delta_{\Sigma}}{2} <  \min \left\lbrace \frac{\epsilon}{\sqrt{d}(4+8 c_{\Sigma^{*}})},\frac{\epsilon}{2 c_{C}} \right \rbrace
\end{align*}
  with probability at least $1 -  c_{2} \exp \left( - \tau_{1} p \right)$. Then $\opnorm{(\Delta_{\Gamma})_{\cdot S}}{\infty} < \epsilon$ and $\|\Delta_g\|_{\infty} < \epsilon$,
using Lemma~\ref{lem:express error in delta}.
\item[3)] We verify that $\epsilon \leq \frac{\alpha}{6c_{\Gamma^{*}}}$ under the assumption on the sample size. Then, the result follows from Theorem~\ref{thm:deterministicres}.
\end{itemize}
In the following, we go through the steps in detail. 
\begin{itemize}
    \item[1)] Using the lower bound on the sample size, it holds that
    \begin{align*}
    \frac{\epsilon}{\sqrt{d}(4+8 c_{\Sigma^{*}})} &= \frac{\sqrt{\tau_1 \tilde{c} d p /n}}{\sqrt{d}(4+8 c_{\Sigma^{*}})}\\
    &< \frac{ \sqrt{\tau_1 \tilde{c} d p /\tau_{1} \tilde{c} d p  \max \lbrace c_{*}^2, 1/4 \rbrace }}{\sqrt{d}(4+8 c_{\Sigma^{*}})}\\
    &= \frac{\sqrt{1/ \max \lbrace c_{*}^2, 1/4 \rbrace}}{{\sqrt{d}(4+8 c_{\Sigma^{*}})}}\\
    & \leq \sqrt{1/ \max \lbrace c_{*}^2, 1/4 \rbrace}\\
    & \leq \sqrt{4} =2. 
    \end{align*}
Using $\tau_{1} \geq 1$, we obtain
\begin{align*}
    \frac{\epsilon}{\sqrt{d}(4+8 c_{\Sigma^{*}})} &=  \frac{\sqrt{\tau_1 \tilde{c} d p /n}}{{\sqrt{d}(4+8 c_{\Sigma^{*}})}}\\
     & > \frac{\sqrt{\tilde{c}} \sqrt{p/n}}{(4+8 c_{\Sigma^{*}})} \\
     & \geq \frac{\sqrt{ (4+8 c_{\Sigma^{*}})^2 16 c_{1}^2 c_{\Sigma^{*}}^2} \sqrt{ p/n }}{(4+8 c_{\Sigma^{*}})} \\
     & = 4 c_{1} c_{\Sigma^{*}} \sqrt{p/n}. 
\end{align*}
\item[2)]  Using Corollary~\ref{cor:spectralcor} we obtain
\begin{align*}
    &\mathbb{P}\left(\opnorm{\Delta_{\Sigma}}{2} \geq \frac{\epsilon}{\sqrt{d}(4+8 c_{\Sigma^{*}})}\right) \\
    \leq & c_2 \exp \left( - \frac{c_{3}}{4 \max(1, c_{\Sigma^{*}}^2)} n \frac{{\tau_{1}}\tilde{c}d p /n}{d(4+8 c_{\Sigma^{*}})^2} \right)\\
    \leq & c_2 \exp \left( - \frac{c_{3} \tilde{c}}{4 \max(1, c_{\Sigma^{*}}^2) (4+8 c_{\Sigma^{*}})^2} \tau_{1} p   \right)\\
    \leq & c_{2} \exp \left( - \tau_{1} p \right)
\end{align*}
\item[3)] We verify that $\epsilon \leq \frac{\alpha}{6c_{\Gamma^{*}}}$ under the assumption on the sample size.
\begin{align*}
  \epsilon = & \sqrt{\tau_1 \tilde{c} d p /n}\\
    \leq & \sqrt{\tau_1 \tilde{c} d p /\tau_1 \tilde{c} d p \max \lbrace c_{*}^2, 1/4 \rbrace } \\
    = & \sqrt{1/\max \lbrace c_{*}^2, 1/4 \rbrace}\\
    \leq & \frac{\alpha}{6c_{\Gamma^{*}}}
\end{align*}
\end{itemize}

For the same choice of $\epsilon$ and $\epsilon /2 c_{C}$ steps 1) - 3) can be carried out analogously and we obtain that 
\begin{align*}
\mathbb{P}\left(\opnorm{\Delta_{\Sigma}}{2} \geq \frac{\epsilon}{2 c_{C}}\right) \leq c_{2} \exp \left( -\tau_{1} dp \right).
\end{align*}
The result follows by applying Theorem \ref{thm:deterministicres}. 
\end{proof}

\section{Irrepresentability Condition}
\label{sec:irrappendix}

This section is divided into four parts. First, we give the proof of Theorem \ref{thm:oderingdiagonal} and provide an illustrating Example. Second, we discuss a weaker notion of the irrepresentability condition \eqref{eq:irrcondition} that is necessary for support recovery and more often fulfilled. Third, we provide a detailed simulation study comparing the fulfillment of the irrepresentability condition and its weaker notion. Finally, we show that the impact of the weak irrepresentability condition is already increasing the performance of the Direct Lyapunov Lasso drastically. 

\subsection{Irrepresentability Proof and Example}
\label{sec:irrproofandexample}
\begin{proof}[Theorem~\ref{thm:oderingdiagonal}]
  Let $\Sigma^0 = \Sigma(M^0, C)$ be the
   covariance matrix  associated to the drift matrix
   $M^0$.  As we are assuming that $C=2I_p$, we have
   \begin{align*}
     \Sigma^{0}=-(M^0)^{-1}=\text{diag}(1/d_1,\dots,1/d_p). 
   \end{align*}
   Writing $\Gamma^0=\Gamma(\Sigma^0)$ for the resulting Gram matrix,
   we define the \emph{local} irrepresentability constant
   \[
     \tilde\rho_G(M^0)=\opnorm{\Gamma^{0}_{S_G^cS_G} (\Gamma^{0}_{S_GS_G})^{-1}}{\infty}.
   \]
   If a small open ball around $M^0$ contains a matrix $M$, then the
   ball also contains all matrices that are obtained from $M$ by
   negating one or more of the off-diagonal entries.  Hence, by
   continuity, the irrepresentability condition for support 
   $S_G$ holds uniformly over a neighborhood of $M^0$  if and only if
   (i) the submatrix $\Gamma^0_{S_GS_G}=(\Gamma^0)_{S_GS_G}$ is invertible and
   (ii) $\tilde\rho_G(M^0)<1$.

   Since $\Sigma^{0}$ is diagonal, plugging it into the coefficient matrix
   from~\eqref{eq:ASigma} gives a symmetric matrix with entries
\begin{align*}
    A(\Sigma^{0})_{(i,j),(k,l)}=
    \begin{cases}
    2/d_{l} & \quad \text{if } i=j=k=l,\\
    1/d_{l} & \quad \text{if } i=k,\, j=l \text{ and } k\neq l,\\
    1/d_{l} & \quad \text{if } i=l,\, j=k \text{ and } k\neq l,\\
    0 & \quad \text{otherwise}.
    \end{cases}
\end{align*}
The entries of the Gram matrix $\Gamma^{0}=\Gamma(\Sigma^{0})$ are the
inner products of the columns of $A(\Sigma^{0})$.  That is,
\begin{align*}
    \Gamma^{0}_{(i,j),(k,l)}=
    \begin{cases}
    4/d_{l}^2 & \quad \text{if } i=j=k=l,\\
    2/d_{l}^2 & \quad \text{if } i=k,\, j=l \text{ and } k\neq l,\\
    2/(d_{k}d_{l}) & \quad \text{if } i=l,\, j=k \text{ and } k\neq l,\\
    0 & \quad \text{otherwise}.
    \end{cases}
\end{align*}
Note that the only off-diagonal entries in $\Gamma^{0}$ occur when the row index
is $(i,j)$ and the column index is $(j,i)$ with $i \neq j$.  We
display the matrices $A(\Sigma^0)$ and $\Gamma^0$ for a graph with
$p=3$ nodes
in Example~\ref{ex:diagM0:p3}.

\emph{Case I: Graph contains a two-cycle}.  Suppose $G$ contains a
two-cycle, say $k\to l\to k$ with $k\not=l$.  The two edges on the cycle index two
columns of $A(\Sigma^0)$ that are linearly dependent.  Indeed, the column
indexed by $(k,l)$ has only two 
nonzero entries in rows $(k,l)$ and $(l,k)$, both of which are equal to $d_l$, and the same holds
for the column indexed 
$(l,k)$ except that the common value of its two nonzero entries is
$d_k$.   The columns $(k,l)$ and $(l,k)$ of $\Gamma^0$ are
similarly linearly dependent.  Hence, the submatrix
$\Gamma^{0}_{S_GS_G}$ fails to be invertible, if the graph $G$
contains a two-cycle.  Consequently, the irrepresentability condition
holds uniformly over a neighborhood of $M^0$ only if $G$ is free of
two-cycles, in which case we call $G$ \emph{simple}.

\emph{Case II. Graph is simple}.  In the rest of the proof suppose
that $G$ is simple.  In this case, the submatrix 
$\Gamma^{0}_{S_GS_G}$ is diagonal with entries
\[
  \Gamma^{0}_{(k,l),(k,l)} =
  \begin{cases}
    4/d_{l}^2 & \quad \text{if } k=l,\\
    2/d_{l}^2 & \quad \text{if } k\neq l,
  \end{cases}  
\]
where $l\to k$ is an edge of $G$.
% \begin{align*}
% =\text{diag}(\underbrace{4/d_1^2,2/d_1^2,\dots,2/d_1^2}_{p_1},\underbrace{2/d_2^2,4/d_2^2,\dots,2/d_{2}^2}_{p_2},\dots,\underbrace{2/d_{p}^2,\dots,2/d_{p}^2,4/d_{p}^2}_{p_p}),\\
%     &(\Gamma^{0}_{S_GS_G})^{-1}=\text{diag}(\underbrace{d_1^2/4,d_1^2/2,\dots,d_1^2/2}_{p_1},\underbrace{d_2^2/2,d_2^2/4,\dots,d_{2}^2/2}_{p_2},\dots,\underbrace{d_{p}^2/2,\dots,d_{p}^2/2,d_{p}^2/4}_{p_p}),
% \end{align*}
% where $p_i$ is the number of edges outgoing from node $i$, including the self-loop.
The second submatrix of interest, $\Gamma^{0}_{S_{G}^{c} S_{G}}$, also
has only one nonzero entry in each column.  If $l\to k$ is an edge,
indexing column $(k,l)$, then the entry is
\begin{align*}
  (\Gamma^{0}_{S_{G}^{c} S_{G}})_{(l,k),(k,l)}= 2/(d_{k} d_{l}).
\end{align*}
Note that $G$ being simple implies that $k\to l$ is not an edge of
$G$.  Multiplying the second submatrix to the inverse of the first, we obtain that
% Note that there is at most one nonzero element per row of $(\Gamma^{0}_{S_{G}^{c} S_{G}})$. When calculating $\Gamma^{0}_{S_G^{c}S_G} (\Gamma^{0}_{S_GS_G})^{-1}$ we multiply this matrix by a diagonal matrix, and hence still obtain at most one nonzero element per row. That is,
\begin{align*}
&(\Gamma^{0}_{S_G^{c}S_G} (\Gamma^{0}_{S_GS_G})^{-1})_{(i,j),(l,k)}\\
    &= \begin{cases}
        d_{k}/d_{l} & \quad  \text{if } (i,j)=(k,l) \text{ and } \,(l,k) \in S_{G},\,(k,l) \in S_{G}^{c}, \\
        0 & \quad \text{otherwise}.
    \end{cases}
\end{align*}
Since $\tilde{\rho}_G(M^0)$ is obtained via the maximum absolute row
sum, we have \linebreak $\tilde{\rho}_G(M^0) < 1$ if and only if $d_i/d_j <1$ for
all pairs $(j,i)\in S_G$, or equivalently, all edges $i \to j\in E$, as the theorem claims. If $G$ contains a cycle of at least length 3, there exists a sequence of edges in $E$ such that $i_{1}\to i_{2} \to i_{3} \to i_{m}\to i_{1}$ with $i_{1}, \dots, i_{m} \in V$. Then, we have $\tilde{\rho}_G(M^0) < 1$ if and only if 
\begin{align*}
d_{i_1}/d_{i_2} < 1, \quad  d_{i_2}/d_{i_3} < 1, \quad \dots \quad d_{i_{m-1}}/d_{i_{m}} < 1, \quad d_{i_m}/d_{i_{1}} < 1.
\end{align*}
Multiplying yields
\begin{align*}
d_{i_1}/d_{i_2} \cdot  d_{i_2}/d_{i_3} \cdot \, \dots \, \cdot d_{i_{m-1}}/d_{i_{m}} \cdot d_{i_m}/d_{i_{1}} =1
\end{align*}
which contradicts that all individual quotients are smaller than one. 
\end{proof}

% \begin{definition}
%  We call a graph $G=(V,E)$ simple if it contains no 2-cycle. This means that for two nodes $i,j \in V$ with $i \neq j$ it is not possible that both $i \to j$ and $j \to i$ are in $E$. 
% \end{definition}

We illustrate the matrix calculations in the proof of
Theorem~\ref{thm:oderingdiagonal} for a graph on $p=3$ nodes.
\begin{example}
\label{ex:diagM0:p3}
We consider the 3-chain $G=(V,E)$ displayed in
Figure~\ref{fig:examplegraph}, and the matrices
\begin{align*}
    M^0&=\text{diag}(-d_1,-d_2,-d_3) \quad\text{and}\quad
\Sigma^{0}=\text{diag}(1/d_1,1/d_2,1/d_3).
\end{align*}
Ordering rows as \\
$(1,1),(1,2),(1,3),(2,1),(2,2),(2,3),(3,1),(3,2),(3,3)$
and columns as \\
$(1,1),(2,1),(3,1),(1,2),(2,2),(3,2),(1,3),(2,3),(3,3)$, we find
\begin{align*}
A(\Sigma^{0})=
\begin{pmatrix}
            2/d_1 &0&0&0 & 0 &0 & 0 & 0 & 0 \\
            0 & 1/d_1 &0& 1/d_2  & 0 &0& 0 & 0 &0\\
            0 &0&1/d_1& 0 & 0 & 0& 1/d_3 & 0 & 0\\
            0 & 1/d_1 &0& 1/d_2  & 0 &0& 0 & 0 &0\\
            0 &0& 0& 0 & 2 /d_2 &0& 0 & 0 & 0 \\
            0 & 0& 0& 0 & 0& 1/d_2   & 0 & 1/d_3 & 0\\
            0 &0&1/d_1& 0 & 0 & 0& 1/d_3 & 0 & 0\\
            0 & 0& 0& 0 & 0& 1/d_2   & 0 & 1/d_3 & 0\\
            0 & 0& 0& 0 & 0& 0 & 0 & 0 & 2/d_3\\
\end{pmatrix}            
\end{align*}
and for $\Gamma^{0}$ using the labelling $(1,1),(2,1),(3,1),(1,2),(2,2),(3,2),(1,3),(2,3),(3,3)$ both for rows and columns we obtain
\begin{align*}
\begin{pmatrix}
            4/d_1^2 &0&0&0 & 0 &0 & 0 & 0 & 0 \\
            0 & 2/d_1^{2} &0& 2/d_1 d_2  & 0 &0& 0 & 0 &0\\
            0 &0&2/d_1^{2}& 0 & 0 & 0& 2/d_1 d_3 & 0 & 0\\
            0 & 2/d_1 d_2 &0&  2/d_2^{2} & 0 &0& 0 & 0 &0\\
            0 &0& 0& 0 & 4 /d_2^{2} &0& 0 & 0 & 0 \\
            0 & 0& 0& 0 & 0& 2/d_2^{2}   & 0 & 2/d_2 d_3 & 0\\
            0 &0&2/d_1 d_3 & 0 & 0 & 0& 2/d_3^{2}& 0 & 0\\
            0 & 0& 0& 0 & 0& 2/d_2 d_3   & 0 &  2/d_3^{2} & 0\\
            0 & 0& 0& 0 & 0& 0 & 0 & 0 & 4/d_3^{2}\\
\end{pmatrix}.         
\end{align*}
Since 
\begin{align*}
    &S_G=\lbrace (1,1),(2,1),(2,2),(3,2),(3,3) \rbrace \quad \text{and} \\
    &S_G^c= \lbrace (3,1),(1,2),(1,3),(2,3) \rbrace
\end{align*}
we obtain
\begin{align*}
    (\Gamma^{0}_{S_GS_G})^{-1}
    &=\diag(d_1^{2}/4,d_1^{2}/2,d_2^2/4,d_2^2/2,d_3^{2}/4),\\
    \Gamma^{0}_{S_G^cS_G}
    &=
    \begin{pmatrix}
         0 & 0 & 0 & 0 & 0\\
         0 & 2/d_1d_2 & 0 & 0 &0\\
         0 & 0 & 0 & 0 & 0\\
         0 & 0& 0 &2/d_2d_3 &0
    \end{pmatrix},
\intertext{and}    
\Gamma^{0}_{S_G^{c}S_G} (\Gamma^{0}_{S_GS_G})^{-1} &=
\begin{pmatrix}
     0 & 0 & 0 & 0 & 0\\
         0 & d_1/d_2 & 0 & 0 &0\\
         0 & 0 & 0 & 0 & 0\\
         0 & 0& 0 &d_2/d_3 &0
\end{pmatrix}.
\end{align*}
To have $\opnorm{\Gamma^{0}_{S_G^{c}S_G}
  (\Gamma^{0}_{S_GS_G})^{-1}}{\infty}<1$, we need $d_{1}/d_{2}<1$ and
$d_{2}/d_{3}<1$. With the edges $1\to 2$ and $2 \to 3$ present in $G$,
this requirement coincides with the statement of Theorem~\ref{thm:oderingdiagonal}.
\end{example}

\subsection{Necessity of the Weak Irrepresentability Condition}
\label{sec:weakirr}

In Theorem~\ref{thm:probsupport}  we show that the irrepresentability condition 
\begin{equation*}
        \opnorm{\Gamma^{*}_{S^{c}S}(\Gamma^{*}_{SS})^{-1}}{\infty} \leq (1-\alpha), \quad \alpha \in (0,1)
\end{equation*}

is sufficient for model selection consistency. As we show in the subsequent Proposition, a weaker version of the condition is indeed necessary for model selection consistency. 
\begin{definition}
Let $M^{*} \in \mathrm{Stab}_{p}$ and $S=S(M)$ its corresponding support set. Then, the weak irrepresentability condition is fulfilled if    
\begin{equation}
\label{eq:weakirrcond}
        \|\Gamma^{*}_{S^{c}S}(\Gamma^{*}_{SS})^{-1}\sign(\vsM)_{S}\|_{\infty} < 1.
\end{equation}
\end{definition}

We would like to address a small subtlety regarding the relation of the irrepresentability condition to the weak irrepresentability condition.
\begin{remark}
\label{rem:weakreltonormal}
    Consider a matrix $M^{*} \in \mathrm{Stab}_{p}$ fulfilling the irrepresentability condition \eqref{eq:irrcond}, then it also fulfills the weak irrepresentability condition \eqref{eq:weakirrcond}. The reasoning is that by multiplying $\Gamma^{*}_{S^{c}S}(\Gamma^{*}_{SS})^{-1}$ with $\sign(\vsM)_{S}$ the absolute values of the entries in a row of $\Gamma^{*}_{S^{c}S}(\Gamma^{*}_{SS})^{-1}$ are added up in the ``worst case''. By applying $\norm{\cdot}_{\infty}$ the maximum value is chosen. That is exactly what  $\opnorm{\Gamma^{*}_{S^{c}S}(\Gamma^{*}_{SS})^{-1}}{\infty}$ is. 
\end{remark}

If the slightly weaker condition \eqref{eq:weakirrcond} is violated and the entries in the drift matrix fulfill a minimal signal strength condition, we cannot recover the correct support asymptotically.  
\begin{proposition} 
\label{prop:irrnec}
    Consider the setting of Corollary~\ref{thm:probsupport}. Let $M^{*} \in \mathrm{Stab}_{p}$ with $S=S(M^{*})$ such that
    \begin{align*}
        \underset{\underset{(j,k) \in S}{1\leq j <k \leq p}}{\min}|M^{*}_{jk}|>\frac{2c_{\Gamma^{*}}}{2-\alpha}\lambda
    \end{align*}
    holds and that the weak irrepresentability condition \eqref{eq:weakirrcond} is violated. For a fixed positive definite matrix $C$, the equilibrium distribution for $M^{*}$ is given by $\mathcal{N}(0,\Sigma^{*})$. Let $X_1, \dots, X_p \in \mathbb{R}^{p}$ be an i.i.d sample of centered observations and let 
    \begin{equation*}
        \hat{\Sigma}^{n} = \frac{1}{n} \sum_{i=1}^{n} X_{i} X_{i}^{\top}
    \end{equation*}
    be the sample covariance. We denote the estimate obtained by the Direct Lyapunov Lasso \eqref{eq:Frobeniuseq} using $\hat{\Sigma}^{n}$ by $\hat{M}^{n}$.  Then it holds that 
    \begin{align*}
        \mathbb{P}(S(\hat{M}^{n})=S(M^{*})) \rightarrow 0 \qquad \text{ for } n \rightarrow \infty 
    \end{align*}
 \end{proposition}
 \begin{proof}
The proof is based on the proof of Theorem~\ref{thm:deterministicres}. Since the optimization problem \eqref{eq:Frobeniuseq} is convex, the KKT - conditions 
\begin{equation}
\label{eq:KKT2}
    \hat\Gamma^{n} \text{vec}(\hat{M}^{n}) -\hat g^{n} + \lambda \hat{z}^{n} = 0,
\end{equation}
with 
\begin{align*}
    \hat{z}_{(i,j)}^{n}=\begin{cases} 
     \text{sign}(\text{vec}(\hat{M}^{n})_{(i,j)})  &\text{if } \text{vec}(\hat{M}^{n})_{(i,j)} \neq 0,\\
    \in [-1,1]   &\text{if } \text{vec}(\hat{M}^{n})_{(i,j)}=0,
    \end{cases}
\end{align*}
are necessary and sufficient for optimality of $\hat{M}^{n}$. Assume that $S(\hat{M}^{n})=S(M^{*})$. Then, $\hat{M}^{n}$  is the unique solution of the support restricted problem \eqref{eq:supportrestricted} and following the calculations in the proof of Theorem~\ref{thm:deterministicres}, the subgradient $\hat{z}_{S^{c}}^{n}$ is given by 
%   
%\begin{equation*}
    %\label{eq:supportrestricted}
%    \text{vec}(\hat{M}^{n})  = \arg\min_{\vM_{S^c} = 0}\ 
%    \frac{1}{2} \vM^\top  \hGamma^{n} \vM - (\hg^{n})^{\top} \vM +\lambda \norm{\vM}_{1}.
%\end{equation*}
%
% Theorem 3.1. in 
\begin{multline}
    \label{eq:zsc2}
    \hat{z}_{S^{c}}^{n}=
    \frac{1}{\lambda}\left[ 
    -\Gamma_{S^{c}S}^{*}(\Gamma_{SS}^{*})^{-1}((\DGamma^{n})_{SS}\text{vec}(\hat{M}^{n})_S+(\Dg^{n})_S)
    + (\DGamma^{n})_{S^{c}S}\text{vec}(\hat{M}^{n})_S \right.\\
       \left. + (\Dg^{n})_{S^{c}}+\lambda
        \Gamma^{*}_{S^{c}S}(\Gamma^{*}_{SS})^{-1} \sign(\text{vec}(\hat{M}^{n})_S)\right].
\end{multline}
We need $\|\tilde{z}_{S^{c}}^{n}\|_{\infty}<1$ for $\hat{M}^{n}$ to satisfy the KKT-condtions \eqref{eq:KKT2}. Using Lemma \ref{lem:express error in delta} together with Corollary \ref{cor:spectralcor}, we obtain that  $\Dg^{n} \overset{P}{\rightarrow} 0$ and that $\DGamma^{n} \overset{P}{\rightarrow} 0$. Moreover, the inequality \eqref{eq:Mdistance} holds for $n$ large enough for $\hat{M}^{n}$ resulting in 
\begin{align*}
    \norm{\mathrm{vec}(\hat{M}^n)_{S} - \vsM_{S}}_{\infty} \leq \frac{2c_{\Gamma^{*}}}{2-\alpha} \lambda.
\end{align*}
Then, we obtain for the weak irrepresentability condition that 
\begin{align*}
    \Gamma^{*}_{S^{c}S}(\Gamma^{*}_{SS})^{-1} \sign(\mathrm{vec}(\hat{M}^{n})_S) =  \Gamma^{*}_{S^{c}S}(\Gamma^{*}_{SS})^{-1} \sign(\vsM_S).
\end{align*}
 Therefore, we obtain 
\begin{align*}
   \|\tilde{z}_{S^{c}}^{n}\|_{\infty} \overset{P}{\rightarrow}  \|\Gamma^{*}_{S^{c}S}(\Gamma^{*}_{SS})^{-1} \sign(\vsM)_S)\|_{\infty} \geq 1.
\end{align*}
Asymptotically, the subgradient condition is violated for $\hat{M}^{n}$ with $S(\hat{M}^{n})=S(M^{*})$ and probability 1. Hence,
    \begin{align*}
        \mathbb{P}(S(\hat{M}^{n})=S(M^{*})) \rightarrow 0 \qquad \text{ for } n \rightarrow \infty .
    \end{align*}
 \end{proof}

\begin{remark}
    It is also easily possible to construct drift matrices fulfilling the weak irrepresentability condition \eqref{eq:weakirrcond}.
The same construction as in Theorem~\ref{thm:oderingdiagonal} can be used.
%is applicable to the weak irrepresentability condition %\eqref{eq:weakirrcond}.
\end{remark}

\subsection{Simulation Studies: Irrepresentability Condition vs. Weak Irrepresentability Condition}
\label{sec:SimStudiesweakvsnormal}

In this section, we want to answer two urgent questions. 
We have shown that for every DAG there exist non-trivial stable
drift matrices such that the irrepresentability condition \eqref{eq:irrcond} holds. The same is possible for the weak irrepresentability condition \eqref{eq:weakirrcond}. These
signals were constructed to be in a neighborhood of diagonal matrices
whose diagonal entries are ordered in accordance with the topological
ordering of the DAG.  As the size of the graphs increases, this
diagonal ordering becomes more restrictive.  Moreover, there might be
signals that have a different diagonal ordering, but still fulfill the
irrepresentability condition. Therefore, the first question is how often the conditions are fulfilled when selecting random drift matrices according to a predetermined distribution.

%In order to investigate these points,
%we simulate stable signals according to a uniform distribution and
%study how often the condition is fulfilled.  Our simulations also
%consider simple and non-simple cyclic graphs, for which we do not know
%how to give a general construction of signals that satisfy
%irrepresentability.

Given a graph $G=(V,E)$, we generate signals
$M^{*} \in \mathrm{Stab}_p(E)$ by drawing from the uniform
distribution on the subset of matrices in $\mathrm{Stab}_p(E)$ that
have all entries in $[-1,1]$.  The sampling is carried out by rejection
sampling, with rejection of matrices that are not stable.

We consider connected graphs with $p= 2,3,4 $ nodes and at most $p(p+1)/2$ edges.  This
includes all DAGs but also many cyclic graphs.
%Otherwise, the graphs would not be identifiable and hence
%$\Gamma_{SS}$ would not be invertible; see Section
%\ref{sec:identifiability}.
Furthermore, we only consider one labeling of vertices for every
graph.  For every graph, we check for one million simulated signals $M^{*}$ if $\rho(M^{*})<1$ and store the signals that meet the irrepresentability condition \eqref{eq:irrcond}. The frequency of signals that meet the irrepresentability condition is shown in Figure \ref{fig:irr}.

\begin{figure}[t]
    \centering
    \includegraphics[width=1\textwidth]{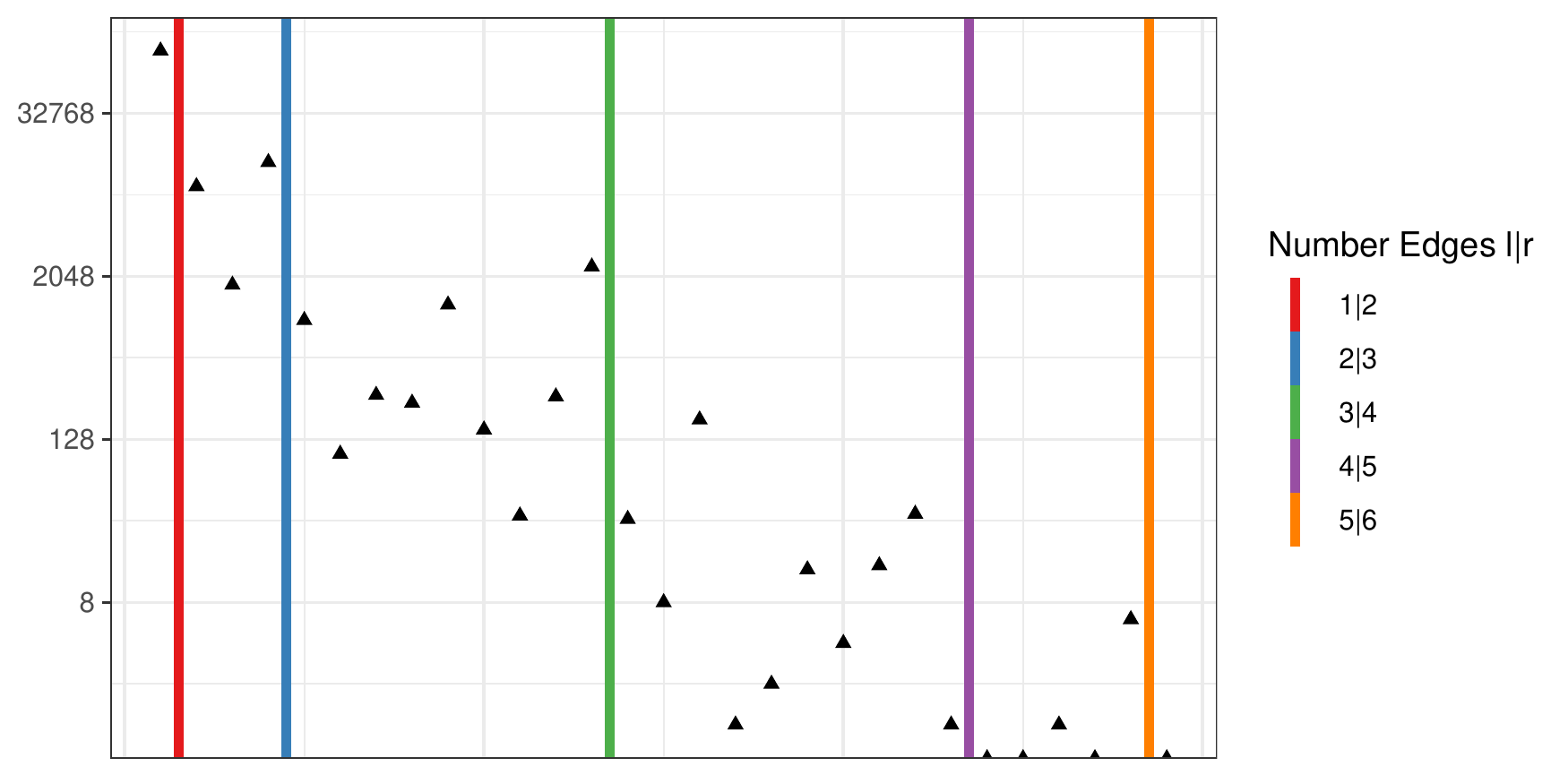}
    \caption{Frequency of the irrepresentability condition \eqref{eq:irrcond} being met
      for one million simulated stable matrices $M^{*}$ for DAGs up to
      4 nodes.  The number of edges is given by the coloring.}
    \label{fig:irrnosign}
\end{figure}

Overall, the frequency with which the irrepresentability condition is fulfilled decreases with an increasing number of edges. The decrease is not monotonic in the number of edges as the restrictiveness is tied to whether an edge adds a new condition on the quotient of the diagonal elements as presented in Theorem~\ref{thm:oderingdiagonal}. An investigation of the drift matrices in Figure~\ref{fig:irrnosign} shows that those fulfilling the irrepresentability condition \eqref{eq:irrcond} all posses a diagonal ordering according to our theoretical result.
\begin{example}
\label{ex:graph3edgesfulirr}
  Consider the graph shown in Figure \ref{fig:graph13}. Drift matrices supported over this graph have the highest frequency of irrepresentability among the graphs
  with three edges in Figure~\ref{fig:irrnosign}. 
  Since there is no edge between the nodes $\lbrace
  1,2,3 \rbrace$, the only conditions on the diagonal are $d_1/d_4<1,
  d_2/d_4 <1$ and $d_3/d_4<1$. Translated this means that $d_{4}$ has to be bigger than $d_{1},d_{2},d_{3}$. 
  \begin{figure}[t]
\centering
\begin{tikzpicture}[->,every node/.style={circle,draw},line width=1pt, node distance=1.5cm, baseline={-.5*\ht\strutbox+.5*\dp\strutbox}]
\vspace*{-1cm}
  \node (1) at (-1,-1) {$1$};
  \node (2) at (1,-1) {$2$};
  \node (3) at (1,1) {$3$};
  \node (4) at (-1,1) {$4$};
\foreach \from/\to in {1/4,2/4,3/4}
\draw (\from) -- (\to);   
\end{tikzpicture}
\caption{The graph on three nodes with highest frequency of simulated
  signals satisfying irrepresentability. % Graph 13 in the list of
  % DAGs.
}
\label{fig:graph13}
\end{figure}

Following Theorem~\ref{thm:deterministicres}, the conditions on the diagonal elements for drift matrices supported over Figure~\ref{fig:Path14} are
$d_{1}/d_{2}<1, d_{2}/d_{3}<1$ and $d_{3}/d_{4}<1$. In particular, these conditions also contain the requirement that 
$d_{4}$ has to be bigger than $d_{1},d_{2},d_{3}$. In addition, they contain the requirement that $d_{3}$ has to be bigger than $d_{1},d_{2}$ and that $d_{2}$ has to be bigger than $d_{1}$. 
\begin{figure}[t]
\begin{center}
\medskip
\begin{tikzpicture}[->,every node/.style={circle,draw},line width=1pt, node distance=1.5cm]
%\hspace*{-1cm}
  \node (1)  {$1$};
  \node (2) [right of=1]{$2$};
  \node (3) [right of=2] {$3$};
  \node (4) [right of=3] {$4$};
\foreach \from/\to in {1/2,2/3,3/4}
\draw (\from) -- (\to);   
 \end{tikzpicture}  
\end{center}
\caption{The path from 1 to 4.}
\label{fig:Path14}
\end{figure}

\end{example}

Another important observation is that the condition is extremely restrictive when selecting stable drift matrices according to a uniform distribution. In Figure~\ref{fig:irrnosign} we observe that already if a graph on 4 nodes has 3 or more edges, the irrepresentability condition is only fulfilled in less than 1 \% of the cases. There even exist some graphs for which the irrepresentability condition is never met. These graphs are displayed in Table~\ref{tab:irrconditioncriticalgraphs}. We tried to find stable drift matrices by applying the above mentioned selection procedure ten million times to these critical graphs. For only two of the graphs we were able to select drift matrices fulfilling the irrepresentability condition. Theorem~\ref{thm:oderingdiagonal} guarantees that there must exist stable drift matrices supported over the two remaining graphs. Using Theorem~\ref{thm:oderingdiagonal}, we put one choice for each of the two graphs in red in Table~\ref{tab:irrconditioncriticalgraphs}.

\begin{longtable}{b{.12\textwidth} | b{.78\textwidth} }
%\begin{table}
%\centering
\caption{\textbf{Left:} The four graphs where none of the one million randomly selected drift matrices $M$ fulfilled the irrepresentability condition in Figure \ref{fig:irrnosign}. \textbf{Right:} Drawing another ten million drift matrices, we obtain for the second and third graph drift matrices that fulfill the irrepresentability condition (black). For the first and fourth graph, we use Theorem \ref{thm:oderingdiagonal} to construct drift matrices that fulfill the irrepresentability condition (red).
\label{tab:irrconditioncriticalgraphs}}
%\begin{tabular}{c|c}

%\centering

\\ 
\hline
%\toprule

\vspace{0.1cm}

\begin{tikzpicture}[->,every node/.style={circle,draw},line width=1pt, node distance=1.5cm, baseline={-.5*\ht\strutbox+.5*\dp\strutbox}]
%\vspace*{-1cm}
  \node[scale=0.8] (1) at (-0.4,-0.4) {$1$};
  \node[scale=0.8] (2) at (0.4,-0.4) {$2$};
  \node[scale=0.8] (3) at (0.4,0.4) {$3$};
  \node[scale=0.8] (4) at (-0.4,0.4){$4$};
\foreach \from/\to in {4/1}
\draw (\from) -- (\to); 
\foreach \from/\to in {4/1,1/2,3/2,4/2,1/3}
\draw[color=black] (\from) -- (\to); 
\end{tikzpicture}

&

\footnotesize
\textcolor{red}{$ 
\begin{pmatrix}
-0.5 & 0 & 0 & 0.05 \\ 
0.05 & -1 & 0.05 & 0.05 \\ 
0.05 & 0 & -0.75 & 0 \\ 
0 & 0 & 0 & -0.25
\end{pmatrix}
$}

\\

\hline
\vspace{0.1cm}

\begin{tikzpicture}[->,every node/.style={circle,draw},line width=1pt, node distance=1.5cm,
      baseline={-.5*\ht\strutbox+.5*\dp\strutbox}
    ]
%\vspace*{1cm}
  \node[scale=0.8] (1) at (-0.4,-0.4) {$1$};
  \node[scale=0.8] (2) at (0.4,-0.4) {$2$};
  \node[scale=0.8] (3) at (0.4,0.4) {$3$};
  \node[scale=0.8] (4) at (-0.4,0.4){$4$};
\foreach \from/\to in {2/1,4/1,4/2,1/3,2/3}
\draw[color=black] (\from) -- (\to); 
\end{tikzpicture}

&\footnotesize
$
\begin{pmatrix}
-0.584860503 &  0.03949857  & 0.0000000 & -0.05605342 \\ 
  0.000000000 & -0.35729470 &  0.0000000 & -0.00303305 \\ 
0.005031837 & -0.08209815 & -0.7782385 &  0.00000000 \\ 
 0.000000000 &  0.00000000 &  0.0000000 & -0.22854795 
\end{pmatrix}
$

\vspace{0.1cm}

\\ 
\hline

\vspace{0.1cm}

\begin{tikzpicture}[->,every node/.style={circle,draw},line width=1pt, node distance=1.5cm, baseline={-.5*\ht\strutbox+.5*\dp\strutbox}]
%\vspace*{-1cm}
  \node[scale=0.8] (1) at (-0.4,-0.4) {$1$};
  \node[scale=0.8] (2) at (0.4,-0.4) {$2$};
  \node[scale=0.8] (3) at (0.4,0.4) {$3$};
  \node[scale=0.8] (4) at (-0.4,0.4){$4$};
\foreach \from/\to in {4/1}
\draw (\from) -- (\to); 
\foreach \from/\to in {3/1,4/1,1/2,4/2,4/3}
\draw[color=black] (\from) -- (\to); 
\end{tikzpicture}

&

\footnotesize
$ 
\begin{pmatrix}
-0.7388917 & 0.0000000 & -0.1277403 & 0.01491351 \\ 
-0.1184546 & -0.9615896 & 0.0000000 & -0.09631827 \\ 
0.0000000 & 0.0000000 & -0.4652617 & 0.04858871 \\ 
0.0000000 & 0.0000000 & 0.0000000 & -0.23807701
\end{pmatrix}
$

\vspace{0.1cm}
\\ 
\hline
\vspace{0.1cm}

\begin{tikzpicture}[->,every node/.style={circle,draw},line width=1pt, node distance=1.5cm, baseline={-.5*\ht\strutbox+.5*\dp\strutbox}]
%\vspace*{-1cm}
  \node[scale=0.8] (1) at (-0.4,-0.4) {$1$};
  \node[scale=0.8] (2) at (0.4,-0.4) {$2$};
  \node[scale=0.8] (3) at (0.4,0.4) {$3$};
  \node[scale=0.8] (4) at (-0.4,0.4){$4$};
\foreach \from/\to in {2/1,3/1,4/1,3/2,4/2,4/3}
\draw (\from) -- (\to); 
\end{tikzpicture}
& 
\footnotesize
\textcolor{red}{$
\begin{pmatrix}
 -1 & -0.05 & 0.05 & 0.05 \\ 
  0 & -0.75 & 0.05 & 0.05 \\ 
  0 & 0 & -0.5 & 0.05 \\ 
  0 & 0 & 0 & -0.25 
\end{pmatrix}
$}
\\
%\hline
\bottomrule
%\end{tabular}
%\end{table}
\end{longtable}

We also carried out the simulation study for simple cyclic graphs. None of the cyclic graphs on 4 nodes fulfilled the irrepresentability condition for ten million randomly selected drift matrices for each graph structure. This is not a proof that the irrepresentability condition \eqref{eq:irrcond} is never met for a cyclic graph, but at least a strong computational evidence. 
In a next step we carry out the same  sampling procedure for graphs on 4 nodes for the weak irrepresentability condition \eqref{eq:weakirrcond} than we did previously for the irrepresentability condition \eqref{eq:irrcond}. The results are displayed in Figure~\ref{fig:irr}.

\begin{figure}[t]
    \centering
    \includegraphics[width=1\textwidth]{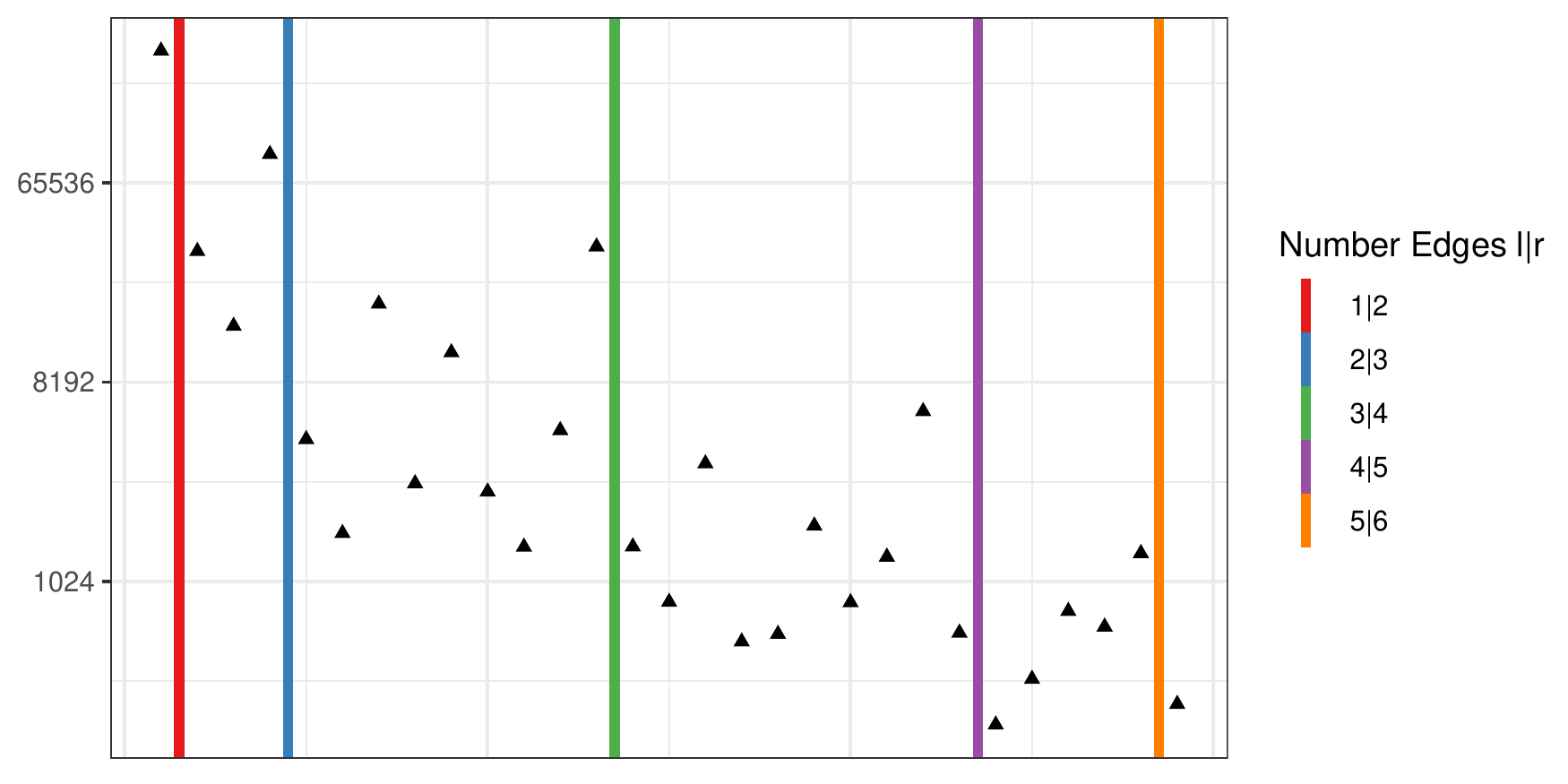}
    \caption{Frequency of the weak irrepresentability condition \eqref{eq:weakirrcond} being met
      for one million simulated stable matrices $M^{*}$ for DAGs up to
      4 nodes.  The number of edges is given by the coloring.}
    \label{fig:irr}
\end{figure}

Comparing the results in Figure~\ref{fig:irr} with those in Figure~\ref{fig:irrnosign}, we observe that the weak irrepresentability condition is much more often fulfilled than the irrepresentability condition. In particular, for all DAGs on 4 nodes there exist drift matrices fulfilling the irrepresentability condition among the one million randomly selected stable drift matrices. The frequency is also much higher. For instance, for graphs with 4 edges, the weak irrepresentability condition is met around 1000 times out of one million while the irrepresentability condition is only met around 50 times out of one million. The reason is that the sign vector in \eqref{eq:weakirrcond} enables fortunate cancellation. 

This comes in handy when carrying out the simulations for simple cyclic graphs. The overall frequency with which one obtains drift matrices supported over simple cyclic graphs fulfilling the weak irrepresentability condition is quite low compared to DAGs.
Nevertheless, we are able to find one drift matrix for every simple cyclic graph on 4 nodes. In Table~\ref{tab:cyclicsignals}, we list all cyclic graphs on 4 nodes together with examples of drift matrices that satisfy the irrepresentability condition.  The selection of graphs includes all graphs that contain at least one directed cycle and are simple (i.e., do not contain a two-cycle).

\newpage

\begin{longtable}{b{.12\textwidth} | b{.78\textwidth} }
%\begin{table}
%\centering
\caption{\textbf{Left:} All simple cyclic graphs with 4 nodes, up to relabelling of the nodes. Edges on cycles are highlighted in red. \textbf{Right:} Specific choice of matrices $M$ matching the graph on the left and fulfilling the weak irrepresentability condition \eqref{eq:weakirrcond}, all entries are rounded to 10 digits.
\label{tab:cyclicsignals}}
%\begin{tabular}{c|c}

%\centering

\\

\hline
\vspace{0.1cm}

\begin{tikzpicture}[->,every node/.style={circle,draw},line width=1pt, node distance=1.5cm,
      baseline={-.5*\ht\strutbox+.5*\dp\strutbox}
    ]
%\vspace*{1cm}
  \node[scale=0.8] (1) at (-0.4,-0.4) {$1$};
  \node[scale=0.8] (2) at (0.4,-0.4) {$2$};
  \node[scale=0.8] (3) at (0.4,0.4) {$3$};
  \node[scale=0.8] (4) at (-0.4,0.4){$4$};
\foreach \from/\to in {2/1,3/2,1/3}
\draw[color=red] (\from) -- (\to); 
\end{tikzpicture}

&\footnotesize
$
\begin{pmatrix}
-0.0444620792 & -0.5733500496 & 0.0000000000 & 0.0000000000 \\ 
 0.0000000000 & -0.0153532191 & 0.0054622865 & 0.0000000000 \\ 
 0.8317033453 & 0.0000000000 & -0.8824298000 & 0.0000000000 \\ 
 0.0000000000 & 0.0000000000 & 0.0000000000 & -0.3405775614  
\end{pmatrix}
$

\vspace{0.1cm}

\\ 
\hline

\vspace{0.1cm}

\begin{tikzpicture}[->,every node/.style={circle,draw},line width=1pt, node distance=1.5cm, baseline={-.5*\ht\strutbox+.5*\dp\strutbox}]
%\vspace*{-1cm}
  \node[scale=0.8] (1) at (-0.4,-0.4) {$1$};
  \node[scale=0.8] (2) at (0.4,-0.4) {$2$};
  \node[scale=0.8] (3) at (0.4,0.4) {$3$};
  \node[scale=0.8] (4) at (-0.4,0.4){$4$};
\foreach \from/\to in {4/1}
\draw (\from) -- (\to); 
\foreach \from/\to in {2/1,3/2,1/3}
\draw[color=red] (\from) -- (\to); 
\end{tikzpicture}

&

\footnotesize
$ 
\begin{pmatrix}
 -0.9780979650 & 0.1042322782 & 0.0000000000 & 0.3752107187 \\ 
  0.0000000000 & -0.7998522464 & -0.4260628200 & 0.0000000000 \\ 
  0.2079165080 & 0.0000000000 & -0.6517819995 & 0.0000000000 \\ 
  0.0000000000 & 0.0000000000 & 0.0000000000 & -0.8112314143
\end{pmatrix}
$

\vspace{0.1cm}
\\ 
\hline
\vspace{0.1cm}

\begin{tikzpicture}[->,every node/.style={circle,draw},line width=1pt, node distance=1.5cm, baseline={-.5*\ht\strutbox+.5*\dp\strutbox}]
%\vspace*{-1cm}
  \node[scale=0.8] (1) at (-0.4,-0.4) {$1$};
  \node[scale=0.8] (2) at (0.4,-0.4) {$2$};
  \node[scale=0.8] (3) at (0.4,0.4) {$3$};
  \node[scale=0.8] (4) at (-0.4,0.4){$4$};
\foreach \from/\to in {1/4}
\draw (\from) -- (\to);   
\foreach \from/\to in {2/1,3/2,1/3}
\draw[color=red] (\from) -- (\to); 
\end{tikzpicture}

& 

\footnotesize
$
\begin{pmatrix}
 -0.6792729949 & -0.6022921619 & 0.0000000000 & 0.0000000000 \\ 
  0.0000000000 & -0.1733464822 & 0.5762203289 & 0.0000000000 \\ 
  0.0383909321 & 0.0000000000 & -0.1785332798 & 0.0000000000 \\ 
  0.2089620568 & 0.0000000000 & 0.0000000000 & -0.6556593408 
\end{pmatrix}
$

\vspace{0.1cm}
\\ 
\hline
\vspace{0.1cm}

\begin{tikzpicture}[->,every node/.style={circle,draw},line width=1pt, node distance=1.5cm, baseline={-.5*\ht\strutbox+.5*\dp\strutbox}]
%\hspace*{-1cm}
  \node[scale=0.8] (1) at (-0.4,-0.4) {$1$};
  \node[scale=0.8] (2) at (0.4,-0.4) {$2$};
  \node[scale=0.8] (3) at (0.4,0.4) {$3$};
  \node[scale=0.8] (4) at (-0.4,0.4){$4$};
\foreach \from/\to in {3/1,2/3,4/2,1/4}
\draw[color=red] (\from) -- (\to); 
\end{tikzpicture}

&

\footnotesize
$
\begin{pmatrix}
 -0.5008390141 & 0.0000000000 & -0.3301411900 & 0.0000000000 \\ 
  0.0000000000 & -0.0754047022 & 0.0000000000 & -0.2224099669 \\ 
  0.0000000000 & 0.9894780936 & -0.8953534714 & 0.0000000000 \\ 
 -0.4568265276 & 0.0000000000 & 0.0000000000 & -0.6545859827 
\end{pmatrix}
$

\vspace{0.1cm}
\\ 
\hline
\vspace{0.1cm}

\begin{tikzpicture}[->,every node/.style={circle,draw},line width=1pt, node distance=1.5cm, baseline={-.5*\ht\strutbox+.5*\dp\strutbox}]
%\hspace*{-1cm}
  \node[scale=0.8] (1) at (-0.4,-0.4) {$1$};
  \node[scale=0.8] (2) at (0.4,-0.4) {$2$};
  \node[scale=0.8] (3) at (0.4,0.4) {$3$};
  \node[scale=0.8] (4) at (-0.4,0.4){$4$};
\foreach \from/\to in {4/1,4/2}
\draw (\from) -- (\to);   
\foreach \from/\to in {2/1,3/2,1/3}
\draw[color=red] (\from) -- (\to); 
 \end{tikzpicture}  

&

\footnotesize
$
\begin{pmatrix}
 -0.9852473154 & 0.0237436080 & 0.0000000000 & -0.1801203806 \\ 
  0.0000000000 & -0.9146776730 & -0.6301784553 & -0.3625553502 \\ 
  0.0314035588 & 0.0000000000 & -0.7371845325 & 0.0000000000 \\ 
  0.0000000000 & 0.0000000000 & 0.0000000000 & -0.2936787312 \\ 
 \end{pmatrix}
$

\vspace{0.1cm}
\\ 
\hline
\vspace{0.1cm}

\begin{tikzpicture}[->,every node/.style={circle,draw},line width=1pt, node distance=1.5cm, baseline={-.5*\ht\strutbox+.5*\dp\strutbox}]
%\hspace*{-1cm}
  \node[scale=0.8] (1) at (-0.4,-0.4) {$1$};
  \node[scale=0.8] (2) at (0.4,-0.4) {$2$};
  \node[scale=0.8] (3) at (0.4,0.4) {$3$};
  \node[scale=0.8] (4) at (-0.4,0.4){$4$};
\foreach \from/\to in {3/2,1/3,4/2,1/4,2/1}
\draw[color=red] (\from) -- (\to); 
\end{tikzpicture}  

&

\footnotesize
$
\begin{pmatrix}
 -0.6168078599 & -0.4643970933 & 0.0000000000 & 0.0000000000 \\ 
  0.0000000000 & -0.8265482867 & 0.0118716909 & 0.4726413568 \\ 
  0.3998511671 & 0.0000000000 & -0.8792877044 & 0.0000000000 \\ 
 -0.5496377517 & 0.0000000000 & 0.0000000000 & -0.7865214688 
\end{pmatrix}
$

\vspace{0.1cm}
\\ 
\hline
\vspace{0.1cm}

\begin{tikzpicture}[->,every node/.style={circle,draw},line width=1pt, node distance=1.5cm, baseline={-.5*\ht\strutbox+.5*\dp\strutbox}]
%\hspace*{-1cm}
  \node[scale=0.8] (1) at (-0.4,-0.4) {$1$};
  \node[scale=0.8] (2) at (0.4,-0.4) {$2$};
  \node[scale=0.8] (3) at (0.4,0.4) {$3$};
  \node[scale=0.8] (4) at (-0.4,0.4){$4$};
\foreach \from/\to in {4/2,1/4}
\draw (\from) -- (\to);   
\foreach \from/\to in {2/1,3/2,1/3}
\draw[color=red] (\from) -- (\to); 
 \end{tikzpicture}  

&

\footnotesize
$
\begin{pmatrix}
-0.2066421132 & -0.0034684981 & 0.1383411973 & 0.0000000000 \\ 
 0.0000000000 & -0.9617960961 & 0.0000000000 & -0.7641737331 \\ 
 0.0000000000 & -0.3169060163 & -0.7561623598 & 0.0000000000 \\ 
 -0.7012514030 & 0.0000000000 & 0.0000000000 & -0.2419070452 \\ 
\end{pmatrix}
$

\vspace{0.1cm}
\\ 
\hline
\vspace{0.1cm}

\begin{tikzpicture}[->,every node/.style={circle,draw},line width=1pt, node distance=1.5cm, baseline={-.5*\ht\strutbox+.5*\dp\strutbox}
]
%\hspace*{-1cm}
  \node[scale=0.8] (1) at (-0.4,-0.4) {$1$};
  \node[scale=0.8] (2) at (0.4,-0.4) {$2$};
  \node[scale=0.8] (3) at (0.4,0.4) {$3$};
  \node[scale=0.8] (4) at (-0.4,0.4){$4$};
\foreach \from/\to in {1/3,2/3}
\draw (\from) -- (\to);   
\foreach \from/\to in {2/1,4/2,1/4}
\draw[color=red] (\from) -- (\to); 
 \end{tikzpicture}  

&

\footnotesize
$
\begin{pmatrix}
-0.8234110032 & -0.6069790549 & 0.0000000000 & 0.0000000000 \\ 
  0.0000000000 & -0.4768311884 & 0.0000000000 & -0.5430481988 \\ 
  -0.1151224086 & 0.5541216009 & -0.8947804412 & 0.0000000000 \\ 
  -0.1818817416 & 0.0000000000 & 0.0000000000 & -0.6244826200 
\end{pmatrix}
$

\vspace{0.1cm}
\\ 
\hline
\vspace{0.1cm}

\begin{tikzpicture}[->,every node/.style={circle,draw},line width=1pt, node distance=1.5cm, baseline={-.5*\ht\strutbox+.5*\dp\strutbox}]
%\hspace*{-1cm}
  \node[scale=0.8] (1) at (-0.4,-0.4) {$1$};
  \node[scale=0.8] (2) at (0.4,-0.4) {$2$};
  \node[scale=0.8] (3) at (0.4,0.4) {$3$};
  \node[scale=0.8] (4) at (-0.4,0.4){$4$};
\foreach \from/\to in {4/2,4/1,4/3}
\draw (\from) -- (\to);   
\foreach \from/\to in {2/1,3/2,1/3}
\draw[color=red] (\from) -- (\to); 
 \end{tikzpicture}  

&

\footnotesize
$
\begin{pmatrix}
-0.7566250684 & 0.1517044385 & 0.0000000000 & 0.0068894741 \\  0.0000000000 & -0.9917302341 & 0.5077337530 & 0.3153799707\\
 0.0895817326 & 0.0000000000 & -0.7472212519 &-0.1730670566\\ 0.0000000000 & 0.0000000000 & 0.0000000000 & -0.3600410065
\end{pmatrix}
$

\vspace{0.1cm}
\\ 
\hline
\vspace{0.1cm}

\begin{tikzpicture}[->,every node/.style={circle,draw},line width=1pt, node distance=1.5cm, baseline={-.5*\ht\strutbox+.5*\dp\strutbox}]
%\hspace*{-1cm}
  \node[scale=0.8] (1) at (-0.4,-0.4) {$1$};
  \node[scale=0.8] (2) at (0.4,-0.4) {$2$};
  \node[scale=0.8] (3) at (0.4,0.4) {$3$};
  \node[scale=0.8] (4) at (-0.4,0.4){$4$};
\foreach \from/\to in {3/1,4/3,1/4,2/1,3/2,4/2}
\draw[color=red] (\from) -- (\to); 
 \end{tikzpicture}  

&

\footnotesize
$
\begin{pmatrix}
-0.8680259003 & 0.4557597358 & -0.0925138230 & 0.0000000000 \\ 
 0.0000000000 & -0.9139470784 & -0.1607573517 & 0.3138186112 \\ 
 0.0000000000 & 0.0000000000 & -0.9212171654 & -0.9521876550 \\ 
 -0.5101859323 & 0.0000000000 & 0.0000000000 & -0.2475099666
\end{pmatrix}
$

\vspace{0.1cm}
\\ 
\hline
\vspace{0.1cm}

\begin{tikzpicture}[->,every node/.style={circle,draw},line width=1pt, node distance=1.5cm, baseline={-.5*\ht\strutbox+.5*\dp\strutbox}]
%\hspace*{-1cm}
  \node[scale=0.8] (1) at (-0.4,-0.4) {$1$};
  \node[scale=0.8] (2) at (0.4,-0.4) {$2$};
  \node[scale=0.8] (3) at (0.4,0.4) {$3$};
  \node[scale=0.8] (4) at (-0.4,0.4){$4$};
\foreach \from/\to in {1/2,3/2,4/2}
\draw (\from) -- (\to);   
\foreach \from/\to in {3/1,4/3,1/4}
\draw[color=red] (\from) -- (\to); 
 \end{tikzpicture}  
 
&
\footnotesize
$
 \begin{pmatrix}
 -0.6688544271 & 0.0000000000 & -0.7215559445 & 0.0000000000 \\ 
 -0.4272868899 & -0.9967063963 & 0.0374428187 & -0.8531300114 \\ 
  0.0000000000 & 0.0000000000 & -0.6779836947 & -0.5781906121 \\ 
 -0.6749138949 & 0.0000000000 & 0.0000000000 & -0.5980373188 
 \end{pmatrix}
$

\\
\hline

%\end{tabular}
%\end{table}
\end{longtable}

The calculations are carried out with the statistic software \texttt{R}. A natural suspicion is that these very few matrices were only selected due to numerical imprecision. In addition, one might wonder if the 10 digits are really necessary. Example~\ref{ex:rationalmatrix} provides more insight using one representative of Table~\ref{tab:cyclicsignals}.
\begin{example}
\label{ex:rationalmatrix}
For the graph in Figure~\ref{fig:3on4nodes} (or first row of Table~\ref{tab:cyclicsignals}) the matrix
\begin{figure}[t]
\centering
\begin{tikzpicture}[->,every node/.style={circle,draw},line width=1pt, node distance=1.5cm,
      baseline={-.5*\ht\strutbox+.5*\dp\strutbox}
    ]
%\vspace*{1cm}
  \node[scale=1] (1) at (-1,-1) {$1$};
  \node[scale=1] (2) at (1,-1) {$2$};
  \node[scale=1] (3) at (1,1) {$3$};
  \node[scale=1] (4) at (-1,1){$4$};
\foreach \from/\to in {2/1,3/2,1/3}
\draw[color=red] (\from) -- (\to); 
\end{tikzpicture} 
\caption{3-cycle in a four node setting.}
\label{fig:3on4nodes}
\end{figure}

\begin{align*}
M=
\begin{pmatrix}
-0.0444620792 & -0.5733500496 & 0.0000000000 & 0.0000000000 \\ 
 0.0000000000 & -0.0153532191 & 0.0054622865 & 0.0000000000 \\ 
 0.8317033453 & 0.0000000000 & -0.8824298000 & 0.0000000000 \\ 
 0.0000000000 & 0.0000000000 & 0.0000000000 & -0.3405775614  
\end{pmatrix}
\end{align*}
fulfills the weak irrepresentability condition. The margins to satisfy the weak irrepresentability condition are thin. Rounding the entries of $M$ potentially yields matrices $M$ that do not satisfy the weak irrepresentability condition. The matrix $M$ displayed in this example results in a value for the left side of \eqref{eq:weakirrcond} of 0.9960339 while the 2 - digit version yields a value of 1.011801, i.e. the longer version fulfills the  weak irrepresentability condition while the shorter version does not. This is the reason for the long displays in Table~\ref{tab:cyclicsignals}. However, for the matrix $M$ in this Example, we are able to rationalize the entries with a tolerance of 0.0001 to obtain
\begin{align*}
M_{R}=
\begin{pmatrix}
-2/45 & -43/75 & 0 & 0 \\ 
 0 & -1/65 & 1/183 & 0 \\ 
 84/101 & 0 & -15/17 & 0\\ 
 0 & 0 & 0 & -31/91  
\end{pmatrix}
\end{align*}
fulfilling the weak irrepresentability condition with all calculations being carried out rationally in \texttt{Mathematica} \citep{math2022}. This allays the concern that these matrices only exist due to numerical imprecision in the calculations.
\end{example}

Summing up the situation for simple cyclic graphs, extensive computation was necessary to present one example for every
simple cyclic graph up to four nodes. 
We were not able to discern structure that would suggest how to
construct such examples in general.
Lastly, we provide some insight into the situation with non-simple graphs. Recall that the same construction as in Theorem~\ref{thm:deterministicres} meeting the irrepresentability condition by diagonal ordering was not possible due to $\Gamma^{0}_{SS}$ being not invertible. However, the work by \cite{dettling2022id} shows that there exist drift matrices $M^{*}$ supported over non-simple graphs resulting in $\Gamma^{*}_{SS}$ being invertible. This leaves the possibility for drift matrices fulfilling the irrepresentability condition. 

We carry out the same simulation setup as for DAGs and simple cyclic graphs with the only difference that we randomly select 100000 stable drift matrices instead of one million. The reason is simply the computation time with the information gained from the results being the same. The results are plotted in Figure~\ref{fig:irrns}. If we scale the frequencies by a factor of 10 and compare them with the frequencies in Figure~\ref{fig:irr}, we observe that they are similar for the corresponding number of edges. A more careful investigation of the drift matrices shows that they fulfill the diagonal ordering of Theorem~\ref{thm:deterministicres} for the ``DAG part'' of the graph the drift matrix is supported over. Generally, we think that exploiting this observation, a theoretical result proving that irrepresentability conditions can be met for certain non-simple graphs should be possible and might be an interesting project for further research. 
\begin{figure}[t]
    \centering
    \includegraphics[width=1\textwidth]{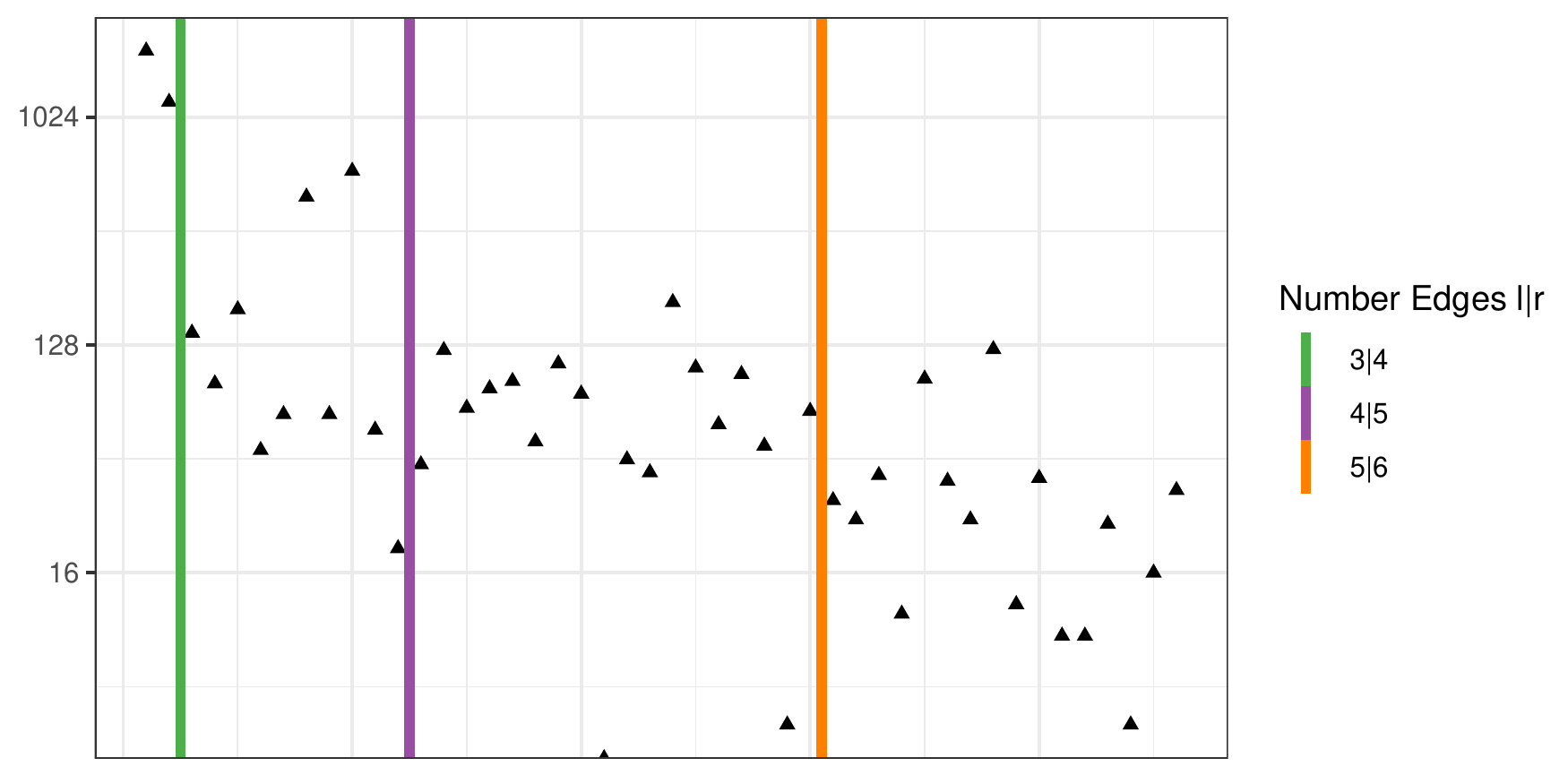}
      \caption{Frequency of the irrepresentability condition being met
        for 100,000 simulated stable matrices $M^{*}$ for non-simple graphs
        without cycles of length $\geq 3$, up to 4 nodes, and up to 6
        edges.
        %Edge weights are drawn from a uniform distribution on
        %$[-1,1]$.
        The number of edges is given by the coloring.}
      \label{fig:irrns}
\end{figure}

\subsection{Simulation Studies: Impact of the Weak Irrepresentability Condition}

Corollary~\ref{thm:probsupport} ensures that if the irrepresentability condition \eqref{eq:irrcond} is fulfilled and some assumptions about minimal signal strength and sample size hold, we are able to recover the support of a drift matrix correctly when applying the Direct Lyapunov Lasso \eqref{eq:Frobeniuseq}. We were not able to  prove this for the weak irrepresentability condition \eqref{eq:weakirrcond}, only its necessity in Proposition~\ref{prop:irrnec} in case a minimal signal requirement is fulfilled. Nevertheless, the condition is quite close to the sufficient condition and fulfilled much more often as we show in Section~\ref{sec:SimStudiesweakvsnormal}. Therefore, we want to investigate the impact of the fulfillment of the weak irrepresentability condition on support recovery. The positive  computational results in this section also translate to the irrepresentability condition as every drift matrix fulfilling the irrepresentability condition also fulfills the weak irrepresentability condition, see Remark~\ref{rem:weakreltonormal}. 

For every DAG on 4 nodes, we select 10 drift matrices fulfilling the weak irrepresentability condition. The selection procedure is the same that we use to obtain Figure~\ref{fig:irr} (uniform distribution of stable matrices with entries
between -1 and 1). Furthermore, we select 100 stable drift matrices supported over the DAGs that do no necessarily fulfill the irrepresentability condition. Based on the drift matrices $M^{*}$ and the Lyapunov equation \eqref{eq:lyapunoveq} with $C= 2 I_{p}$, we calculate the equilibrium covariance matrices $\Sigma^{*}$. We then sample data with $n=100$ from the normal distributions $\mathcal{N}(0,\Sigma^{*})$. Then, we apply the Direct Lyapunov Lasso \eqref{eq:Frobeniuseq} along a regularization path
\begin{align*}
    \lambda_{1} = \lambda_{\max},\dots,\lambda_{100}= \frac{\lambda_{\max}}{10^4}
\end{align*}
where $\lambda_{\max}$ is chosen on an initial grid such that $\hat{M}$ is diagonal.
For the estimates $\hat{M}_{1},\dots,\hat{M}_{100}$ obtained along the regularization path, we calculate some basic metrics  regarding support recovery of the data generating $M^{*}$.
\begin{definition}
\label{def:basicmetrics}
    Let $\hat{M} \in \mathbb{R}^{p \times p}$ be an estimate and let $M^{*}$ be the estimation target. Then, we define 
    \begin{align*}
        tp&= | \lbrace \hat{M}_{ij}: \hat{M}_{ij} \neq 0 \, \text{and} \,  M^{*}_{ij} \neq 0 \rbrace |, \\
        fp&= | \lbrace \hat{M}_{ij}: \hat{M}_{ij} \neq 0 \, \text{and} \,  M^{*}_{ij} = 0 \rbrace |, \\
        tn&= | \lbrace \hat{M}_{ij}: \hat{M}_{ij} = 0 \, \text{and} \,  M^{*}_{ij} = 0 \rbrace |, \\
        fn&= | \lbrace \hat{M}_{ij}: \hat{M}_{ij} = 0 \, \text{and} \,  M^{*}_{ij} \neq 0 \rbrace | .
    \end{align*}
\end{definition}
While these metrics already provide some insights, there exist more refined metrics to evaluate the performance of a structure learning algorithm. 

\begin{definition}
\label{def:refinedmetrics}
     Let $\hat{M} \in \mathbb{R}^{p \times p}$ be an estimate and let $M^{*}$ be the estimation target and let \\
     $tp,fp,tn,fn$ be defined as in Definition~\ref{def:basicmetrics}. Then, we define 
     \begin{align*}
         \textbf{tpr } \text{(true positive rate)}&= \frac{tp}{tp+fn},\\
         \textbf{fpr } \text{(false positive rate)}&= \frac{fp}{fp+tn},\\
         \textbf{acc } \text{(accuracy)}&= \frac{tp+tn}{tp+tn+fp+fn},\\
         \mathbf{f_1}\textbf{-score} &= \frac{2tp}{2tp+fp+fn}, \\
         \textbf{pr } \text{(precision)} &= \frac{tp}{tp+fp}.
     \end{align*}
Calculating tpr and fpr for all regularization parameters, we define the roc curve as plotting tpr vs. fpr with fpr ranging from 0 to 1 using interpolation and extrapolation if necessary. The \textbf{auc roc} or just auc is then defined as the area under the roc curve.  Calculating pr and tpr for all regularization parameters, we define the pr curve as plotting pr vs. tpr with tpr ranging from 0 to 1 using interpolation and extrapolation if necessary. The \textbf{aupr} curve is then defined as the area under the precision curve.   
\end{definition}

For the estimates $\hat{M}_{1},\dots,\hat{M}_{100}$ obtained for each DAG and for each initial drift matrix $M^{*}$, we calculate the metrics mean tpr, mean fpr and max acc, max $f_1$-score. All metrics are then averaged out over the 10 drift matrices fulfilling the weak irrepresentability condition per DAG or over the 100 randomly selected drift matrices, respectively. The results are displayed in Figure~\ref{fig:variousmetirr}. The empty triangles correspond to the average over the randomly selected drift matrices while the full triangles correspond to the average over the drift matrices fulfilling the weak irrepresentability condition. While acc is an average over both correctly classified groups (tp,tn), the $f_1$-score is an average of tpr and precision and focuses more on the correctly identified non-zero entries (tp). The maximum is taken to show the potential of the method if the optimal $\lambda$ is chosen by a model selection method. Additionally, we provide the mean tpr and mean fpr over the regularization path. They are supposed to provide insight into the overall performance of the Direct Lyapunov Lasso over the full regularization path focusing on the correctly or wrongly identified non-zero entries. Despite the weak irrepresentability condition not being proven sufficient, we observe that the max acc and the max $f_1$-score are one for almost all the DAGs. This indicates that there is at least one value of $\lambda$ giving the correct support. Even the few graphs for which these two metrics are not one achieve very high values of around 0.9. Numerical imprecisions or particularly small entries in the data generating signal $M^{*}$ might prevent perfect support recovery for these cases. Contrary to the drift matrices fulfilling the weak irrepresentability condition, the randomly selected signals perform much worse and oftentimes only achieve a value of 0.5 for max acc and max $f_1$-score. That is as good as random guessing. There are a few DAGs with a smaller number of edges where the differences between the two groups is not that severe. The reason is that among these randomly selected drift matrices are more that fulfill the weak irrepresentability condition as the smaller number of edges imposes less contraints on the diagonal ordering, see Theorem~\ref{thm:oderingdiagonal}. Moreover, sparser structures seem to be more favorable when aiming for support recovery. This natural behaviour is also visible in the simulation studies in Section~\ref{sec:robust} and Section~\ref{sec:realworlddata}.

\begin{figure}[t]
    \centering
    \includegraphics[width=1\textwidth]{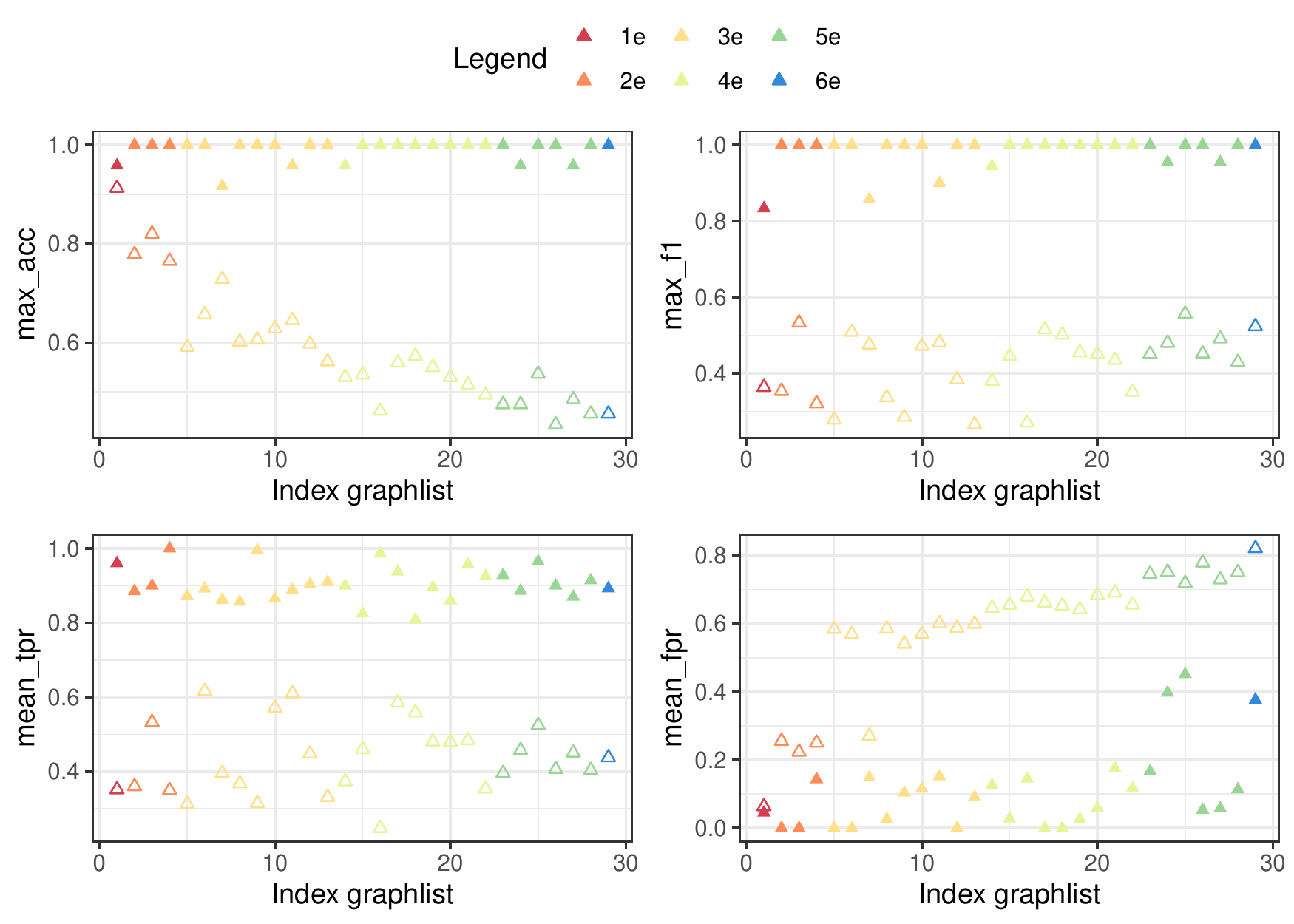}
      \caption{Four metrics measuring the quality of the estimate for DAGs with up to 4 nodes.  The number of edges is given by the coloring. \textbf{Empty:} irrepresentability condition in general not fulfilled, \textbf{Full:} irrepresentability condition fulfilled.}
      \label{fig:variousmetirr}
\end{figure}

Lastly, we present the results for the auc roc using the exact same simulation setup as for Figure~\ref{fig:variousmetirr}. The auc is particularly insightful as the roc curve is obtained by plotting tpr vs. fpr and is in this way returns a ratio of hits to wrong entries. A value of 0.5 means that the method applied  performs badly (is as good as random guessing) while a value of 1 is optimal. For the drift matrices fulfilling the weak irrepresentability condition, we observe that the auc is above 0.9 for almost all graphs that fulfill the weak irrepresentability condition while the performance is very poor for the randomly selected ones. Again as in Figure~\ref{fig:variousmetirr}, we observe that the sparser graphs where probably more randomly selected drift matrices fulfill the weak irrepresentability, perform better than those with more edges.

\begin{figure}[t]
    \centering
    \includegraphics[width=1\textwidth]{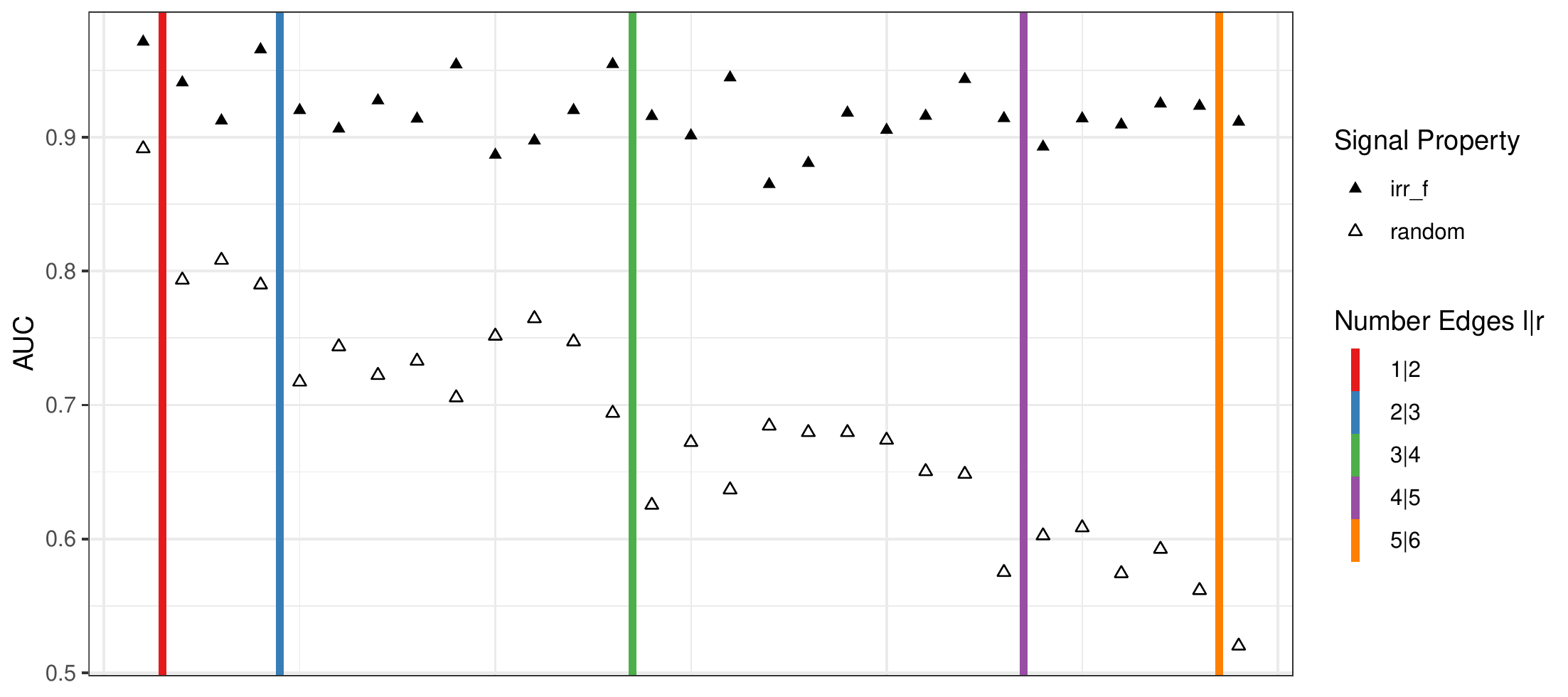}
      \caption{The AUC values for DAGs with up to 4 nodes.  The number of edges is given by the coloring. \textbf{Empty:} irrepresentability condition in general not fulfilled, \textbf{Full:} irrepresentability condition fulfilled.}
      \label{fig:aucroc}
\end{figure}

We did not include further simulations for cyclic graphs in the above setting  which is mainly because we already struggle to find 10 drift matrices supported over cyclic graphs fulfilling the weak irrepresentability condition. In particular, we struggle to find 10 ``really different'' drift matrices that do not only differ in by a small margin in the individual entries. The same applies for the irrepresentability condition and some DAGs. However, all graphs that fulfill the irrepresentability condition, also fulfill the weak irrepresentability condition as we explain in Remark~\ref{rem:weakreltonormal}. In this way they are already included in the simulations in this section and only more positive results are to be expected for graphs fulfilling the irrepresentability.  In fact, the results should be perfect if not for very small entries in the drift matrix or numerical imprecisions (Theorem~\ref{thm:probsupport}).

\newpage

%\section{Sachs Dataset: ``Ground Truth''}
%\begin{figure}[h]
%     \centering\def\figmod{.8}
% \resizebox{7cm}{7cm}{%      
% \begin{tikzpicture}[->,every node/.style={circle,draw},line width=1pt, node distance=1.25cm]
% %\hspace*{-1cm}
%   \node (PIP3) at (5,0)     {PIP3};
%   \node (PIP2) at (4.206268, 2.703204) {PIP2};
%   \node (Plcg) at (2.077075, 4.548160) {Plcg};
%   \node (Mek) at (-0.7115742,  4.9491072) {Mek};
%   \node (Raf) at (-3.274304,  3.778748) {Raf};
%   \node (Jnk) at (-4.797465,  1.408663) {Jnk};
%   \node (P38) at (-4.797465, -1.408663) {P38};
%   \node (PKC) at (-3.274304, -3.778748) {PKC};
%   \node (PKA) at (-0.7115742, -4.9491072) {PKA};
%   \node (Akt) at (2.077075, -4.548160) {Akt};
%   \node (Erk) at (4.206268, -2.703204) {Erk};
  
%   %\node (2) at (1*\figmod,-2*\figmod)     {$2$};
%   %\node (3) at (3*\figmod,-2*\figmod)     {$3$};
%   %\node (4) at (4*\figmod,0)     {$4$};
%   %\node (5) at (2*\figmod,1*\figmod)     {$5$};
% \foreach \from/\to in  {PKA/Jnk,PKC/Jnk,PKA/P38,PKC/P38,PKA/Akt,PIP3/Akt,PKA/Erk,Mek/Erk,PKA/Mek,PKC/Mek,Raf/Mek,
% PKA/Raf,PKC/Raf,PIP2/PKC,Plcg/PKC,PIP3/PIP2,Plcg/PIP2,PIP3/Plcg,PIP3/PKA,PKA/PIP3}
% \draw (\from) -- (\to);  
% \end{tikzpicture}
% }
% \caption{Among scientists accepted signaling molecule interactions for the dataset by \cite{Sachs2005}.}
% \label{fig:SachsdataGT}
% \end{figure}

%\section*{References}

% Acknowledgments---Will not appear in anonymized version

\end{document}